\documentclass[11pt]{article}
\usepackage[latin1]{inputenc}
\usepackage[english]{babel}
\usepackage{amsthm}
\usepackage{amsmath,amsfonts,amscd,amssymb}
\numberwithin{equation}{section}
\usepackage{makeidx}

\usepackage{color}

\theoremstyle{plain}
  \newtheorem{theorem}{Theorem}[section]
  
  \newtheorem{corollary}[theorem]{Corollary}
  \newtheorem{lemma}[theorem]{Lemma}
  \newtheorem{proposition}[theorem]{Proposition}
\theoremstyle{definition}
  \newtheorem{definition}[theorem]{Definition}
  
\theoremstyle{remark}
  \newtheorem{remark}[theorem]{Remark}

\let\noi=\noindent

\makeindex
\title{Key polynomials in dimension 2}
\author{
W. Mahboub\\
Lebanese University, Lebanon. \and
M. Spivakovsky\\
Institut de Math\'ematiques de Toulouse\\
UMR 5219 du CNRS,\\
Universit\'e Paul Sabatier\\
118, route de Narbonne\\
31062 Toulouse cedex 9, France.\\email:
mark.spivakovsky@math.univ-toulouse.fr }

\begin{document}
\newcommand{\m}{\mathfrak{m}}
\newcommand{\R}{\mathbb{R}}
\newcommand{\C}{\mathbb{C}}
\newcommand{\Q}{\mathbb{Q}}
\newcommand{\Z}{\mathbb{Z}}
\newcommand{\N}{\mathbb{N}}
\newcommand{\V}{\mathcal{V}}
\newcommand{\T}{\mathcal{T}}
\newcommand{\g}{\Gamma}
\newcommand{\h}{\Phi}
\newcommand{\ch}{{\operatorname{char}}}
\newcommand{\ini}{{\operatorname{in}}}
\newcommand{\ord}{{\operatorname{ord}}}
\newcommand{\gr}{{\operatorname{gr}}}
\newcommand{\rk}{{\operatorname{rk}}}
\newcommand{\md}{{\operatorname{mod}}}

\renewcommand{\do}{,\ldots ,}

\maketitle
\newpage
\section{Introduction}

Throughout this paper all the rings considered will be commutative with 1.\\
Let $(R,\mathfrak m,k)$ be a regular local ring of dimension $2$ and $F$ the field of fractions of $R$. Consider the poset $(\V,\leq)$ of normalized valuations of $F$ centered at $R$ (see \S\ref{basics}).\\

In \cite{FJ} C. Favre and M. Jonsson prove that $(\V,\leq)$ has the structure of a parametrized, rooted, non-metric tree when
$R=\C[[x,y]]$, the ring of formal power series over the field of complex numbers. The proof of  C. Favre and M. Jonsson is based on associating to each valuation in $\V$ a set of key polynomials, a concept introduced by S. MacLane in \cite{M1} and \cite{M2}. Below we will refer to this set as a {\bf complete set of key polynomials} (see \S\ref{com_key_pol} for its definition and porperties).\\

In \cite{G} A. Granja generalizes this result to the case when $R$ is any two-dimensional regular local ring. A. Granja gives a proof based on associating to each valuation in $\V$ a sequence of point blowing ups.\\
 
In this paper we give a new proof of A. Granja's result when $R$ is any two-dimensional regular local ring, using appropriate complete sequences of key polynomials, based on the work of M. Vaqui\'e \cite{V1} for valuations of arbitrary rank, and the work of F. J. Herrera Govantes, W. Mahboub, M. A. Olalla Acosta and M. Spivakovsky (\cite{HOS}, \cite{HMOS}) for valuations of rank $1$ over fields of arbitrary characteristic.\\

We use the notion of key polynomials introduced in \cite{DSM} and \cite{NS}. We give a simple construction of a complete set of key polynomials associated to a valuation of the field $k(x,y)$ where $k$ is the residue field of $R$ and $x,y$ are independent variables. For explicit constructions of key polynomials on particular cases, see \cite{S}, \cite{FJ}, \cite{K}.\\

We start by stating in \S\ref{basics} the basic facts related to valuations needed in this paper. Then we establish, in \S\ref{blowup}, a natural order-preserving bijection between valuations of $F$ centered at $R$ and valuations of $k(x,y)$ centered at $k[x,y]_{(x,y)}$. This is the content of Theorem \ref{bijection}. It consists of describing a one-to-one correspondence between the set of valuations centered at $R$ and the set of simple sequences of local point blowing ups (see Corollary \ref{val_blow_onetoone}).\\

In \S\ref{com_key_pol} we give the definition of key polynomials. We state the needed facts about key polynomials and construct a complete set of key polynomials associated to a valuation $\nu$ of $k(x,y)$. This is our main tool for the proof of Theorem \ref{tree_structure}. We also define invariants of valuations centered in regular two-dimensional local rings.\\

Then we consider two comparable valuations, $\mu\leq \nu$, and study the structure of their key polynomials sets and the relation between the invariants of those valuations. This is done in \S\ref{order_val}. Using this comparison,  we prove that the infimum of any two elements of $\V$ exists (Theorem \ref{infimum}) and that any increasing sequence in $\V$ has a majorant in $\V$ (Theorem \ref{majorant}). We note that a more general version of the latter result --- one for rings of arbitrary dimension --- is given in Lemma 3.9 (i) of \cite{N}.\\

Finally, in \S\ref{trees} we prove the main theorem of this paper, Theorem \ref{tree_structure}. This Theorem asserts that $\V$ has a tree structure.

We thank the referee for a very careful reading of the paper and for numerous useful comments that helped improve the exposition.

\section{Basics}\label{basics}

Let $R$ be a regular noetherian local ring of dimension 2. Denote by $\m$ its maximal ideal and let $F$ be the quotient field of $R$.

A \textbf{valuation} of $F$ is a function $\nu:F\longrightarrow \bar{\R}=\R\cup \{\infty\}$ such that for all $f,g\in F$: 
\begin{description}
\item[$(V_1)$] $\nu(f+g)\geq\inf(\nu(f),\nu(g))$,
\item[$(V_2)$] $\nu(f\cdot g)=\nu(f)+\nu(g)$.
\end{description}
It is an easy exercise to check that if $\nu$ is not constant, then axiom $(V_2)$ implies
\smallskip

\noindent$(V_3)$ $\nu(1)=0$.

Let $\Gamma$ be a totally ordered abelian group. A \textbf{Krull valuation} of $F$ is a function
$$
\nu:F\longrightarrow \Gamma\cup\{\infty\}
$$
satisfying $(V_1)$, $(V_2)$ and $(V_3)$ such that $\nu^{-1}(\infty)=0$.\\

If $\nu$ is a valuation or a Krull valuation of $F$, we say that $\nu$ is \textbf{centered} at $R$ if $\nu$ is non-negative on $R$ and strictly positive on $\m$. We say that $\nu$ is \textbf{proper} if $\nu(F\setminus\{0\})\neq\{0\}$ and $\nu(\m)\neq\{\infty\}$.\\

If $\nu$ and $\nu'$ are two valuations of $F$, then we say that $\nu$ and $\nu'$ are \textbf{equivalent}, and write $\nu\sim\nu'$, if there exists a non-zero real number $c$ such that for all $f\in F$ we have $\nu(f)=c\nu'(f)$.\\

Let $\V=\{\nu\mid\nu\ proper\ valuation\ centered\ at\ R\}/\sim$. When working with an element of $\V$, we will tacitly fix a valuation representing it, so in practice we will work with valuations instead of classes of valuations. We will consider only normalized valuations, in the sense that $\nu(\m):=inf\{\nu(f)\ |\ f\in\m\}=1$. Indeed, we can represent any element $\nu$ of $\V$ by a uniquely determined normalized valuation after multiplying all the values by $\frac1{\nu(\m)}$.\\
 
For an element $\nu\in \V$ we will denote by $\Gamma_{\nu}$ the \textit{augmented} value group of $\nu$, that is,
$\Gamma_{\nu}:=\nu(F)\subset\bar{\R}$.\\ 
 
If $\nu$ is a valuation (resp. a  Krull valuation), the set 
$$
R_\nu:=\left\{f\in F\ |\ \nu(f)\geq 0\right\};
$$
is a local ring called the \textbf{valuation ring associated to} $\nu$, with maximal ideal
$$
\m_\nu:=\left\{f\in F\ |\ \nu(f)> 0\right\}.
$$
The \textbf{rank} of $\nu$, denoted by $\rk(\nu)$, is the Krull dimension of $R_{\nu}$. In our situation $\rk(\nu)$ is at most $2$ by Abhyankar's inequality.\\
\begin{remark} We have $\rk(\nu)=1$ if and only if $\nu(F\setminus\{0\})\subset\R$ (resp. the group $\nu(F\setminus\{0\})$ can be embedded into the additive group $\R$ of real numbers).
\end{remark}
\medskip

If $S$ is a ring contatined in $R_{\nu}$, the \textbf{center} of $\nu$ in $S$ is the prime ideal $\m:=\m_{\nu}\cap S$. In this situation we also say that $\nu$ \textbf{is centered at }$\m$. If $(S,\m)$ is a local domain, we will sometimes say that $\nu$ \textbf{is centered at }$(S,\m)$.\\

\begin{definition} For two local rings $(R_1,\m_1)$ and $(R_2,\m_2)$, we say that $R_2$ \textbf{dominates} $R_1$ if $R_1\subset R_2$ and $\m_1=R_1\cap \m_2$. If, in addition, $R_1$ and $R_2$ are domains with the same field of fractions, we will say that $R_2$ \textbf{birationally dominates} $R_1$
\end{definition}
\noindent\textbf{Notation.} In the above situation we will write $(R_1,\m_1)<(R_2,\m_2)$ or simply $R_1<R_2$.
\medskip

\begin{remark} Let $(S,\m)$ be a local domain, contained in $F$. A valuation $\nu$ of $F$ is centered at $\m$ if and only if we have
$(S,\m)<(R_\nu,\m_\nu)$. 
\end{remark}
\begin{remark}\label{characterizingvaluaitonrings} (1) The valuation $\nu$ is uniquely determined by its valuation ring $R_\nu$. For a proof, see \cite{V0}, Proposition 1.4.
\medskip

(2) Consider a local domain $(R,\m)$ with field of fractions $F$. The following conditions are equivalent:

(a) $(R,\m)$ is of the form $(R_\nu,\m_\nu)$ for some valuation $\nu$ of $F$

(b) for every $f\in F$ either $f\in R_\nu$ or $f^{-1}\in R_\nu$ (or both)

(c) the ring $(R,\m)$ is maximal with respect to the relation of birational domination.
\medskip

For a proof, see \cite{V0}, Proposition 1.4 and \cite{Bo}, Chap. 6, \S2, n$^0$2, Th\'eor\`eme 1, page 85.

Below we reprove the equivalence (a)$\Longleftrightarrow$(c) in the special case when $R$ is a 2-dimensional regular local ring. 
\end{remark}
\medskip

We will use the following partial order on the set of valuations of the field $F$, centered at $R$:\\
For two valuations $\mu$ and $\mu'$ centered at $R$, we will say that $\mu\leq \mu'$ if $\mu(f)\leq\mu'(f)$ for all $f\in R$.\\

We denote by $\tilde{\nu}_{\m}$ the multiplicity valuation, that is, $\tilde{\nu}_{\m}(f):=max\{i\mid\ f\in\m^{i}\}$ for all $f\in R$. We note that
$\tilde{\nu}_{\m}$ is the smallest element of $\V$. We say that the  \textbf{multiplicity} of $f$ at $\m$ is $\tilde{\nu}_{\m}(f)$.\\

If $\nu$ is a valuation centered at $R$ then $\nu$ determines a Krull valuation $\nu'$, centered at $R$. Furthermore, $\rk(\nu')=2$ if and only if $\nu^{-1}(\infty)\neq\{0\}$. Indeed, if $\nu^{-1}(\infty)=\{0\}$, then $\nu=\nu'$ is a Krull valuation of rank $1$. Otherwise, if $\nu^{-1}(\infty)\neq\{0\}$, then $\nu^{-1}(\infty)$ is a principal prime ideal of $R$ generated by an irreducible element $f\in R$. For each $g\in R-\{0\}$, write $g=f^{s}h$, where $f$ does not divide $h$ (that is, $\nu(h)< \infty$). Define
$\nu'(g)=(s,\nu(h))\in \Z\oplus \R$. For $G=\frac{g_1}{g_2}\in F-\{0\}$, put $\nu(G)=\nu(g_1)-\nu(g_2)$. It is clear that $\nu'$ determines a Krull valuation on $F$, centered at $R$.

Conversely, if $\nu'$ is a  Krull valuation of $F$ centered at $R$ then $\nu'$ determines a valuation on $R$. Indeed, let $\m_{\nu'}$ be the maximal ideal of $R_{\nu'}$. If $\rk(\nu')=1$, put $\nu=\nu'$. Otherwise, let $\Gamma_1$ be the isolated subgroup of
$\Gamma$ of rank $1$ (that is, the smallest non-zero isolated subgroup of $\Gamma$). Let $P'$ be the prime ideal of $R_{\nu'}$ associated to $\Gamma_1$:
$$
P'=\left\{\left.f\in R_{\nu'}\ \right|\ \nu(f)\in\Gamma\setminus\Gamma_1\right\}.
$$
Let $P=P'\cap R$. For each $f\in R$, if $f\in P$, put $\nu(f)=\infty$. Otherwise, put
$$
\nu(f)=\nu'(f)\in\Gamma_1.
$$
Then $\nu$ is a valuation centered at $R$.

For an element $\beta\in\Gamma_{\nu}$, let $\mathbf{P}_{\beta}:=\left\{f\in F\ |\ \nu(f)\geq \beta\right\}$,
$\mathbf{P}_{\beta+}:=\left\{f\in F\ |\ \nu(f)>\beta\right\}$. Let $\gr_\nu
F=\bigoplus\limits_{\beta\in\Gamma_1}\frac{\mathbf{P}_{\beta}}{\mathbf{P}_{\beta+}}$. For an element $\beta\in\Gamma_1$ and an element $f\in F$ such that $\nu(f)=\beta$, we will denote by $\ini_\nu f$ the natural image of $f$ in
$\frac{\mathbf{P}_{\beta}}{\mathbf{P}_{\beta+}}$.

\section{Valuations and blowing ups}\label{blowup}

The aim of this section is to describe a natural order-preserving bijection between valuations of $F$ centered at $R$ and valuations of $k(x,y)$ centered at $k[x,y]_{(x,y)}$.
\begin{remark} Throughout the paper we will commit the following abuse of notation. We will use the letters $x,y$ to denote both the generators of the field $k(x,y)$ over $k$ and a regular system of parameters of $R$. Since in each case we will specify clearly with which ring we are working, this should cause no confusion.
\end{remark} 
A \textbf{simple} sequence $\pi^*$ of local point blowings up of $Spec\ R$ is a sequence of the form
\begin{equation}
(R,\m)\overset{\pi_1}\longrightarrow(R_1,\m_1)\overset{\pi_2}\longrightarrow\dots\overset{\pi_i}\longrightarrow(R_i,\m_i)
\overset{\pi_{i+1}}\longrightarrow\dots\label{eq:sequence}
\end{equation}
where $\pi_i$ is given by considering the blowing up $Spec\ R_{i}\overset{\pi^*_i}\longrightarrow Spec\ R_{i-1}$ along $\m_{i-1}$, picking a point $\xi_i\in{\pi^*_i}^{-1}(\m_{i-1})$ and putting $R_i:=\mathcal O_{X_i,\xi_i}$. Let $\Pi(R)$ denote the set of all the simple sequences (finite or infinite) of local point blowings up  of $Spec\ R$. Fix 
an element $f\in R\setminus\{0\}$.
Let $\mu$ denote the multiplicity of $f$ at $\m$. Assume that
\begin{equation}
f\notin\left(x,y^{\mu+1}\right).\label{eq:ymu}
\end{equation}

\begin{definition}
A monomial ideal in a regular local ring $A$, with regular system of parameters $(u_1,\dots,u_s)$ is an ideal in $A$ generated by monomials in $(u_1,\dots,u_s)$.
\end{definition}

Let $I(x,y,f)$ denote the smallest \textit{monomial} ideal containing $f$. (\ref{eq:ymu}) is equivalent to saying that $y^\mu\in
I(x,y,f)$.

Let $e(x,y,f):=\min\left\{\left.\frac\alpha{\mu-\beta}\ \right|\ x^\alpha y^\beta\in I(x,y,f),\beta<\mu\right\}\in\frac1{\mu!}\mathbb
Z\cup\{\infty\}$, where we adopt the convention that the minimum of the empty set is infinity.
\begin{definition} The \textbf{first characteristic exponent} of $f$ at $\m$ is the supremum of $e(x,y,f)$, where $(x,y)$ runs over all the regular systems of parameters of $R$ satisfying (\ref{eq:ymu}).
\end{definition}
Fix a real number $e$. For a real number $\xi$, let $I_\xi$ denote the monomial ideal of $R$ generated by all the monomials $x^\alpha y^\beta$ such that $\alpha+e\beta\ge \xi$.
\begin{definition} The \textbf{monomial} valuation $\nu_{x,y,e}$ of $R$, associated to the data $(x,y,e)$ is the valuation defned by $\nu_{x,y,e}(g)=\max\{\xi\in\mathbb R\ |\ g\in I_\xi\}$.
\end{definition}

Let $R^{*}$ denote the set of units in $R$.

\begin{proposition}\label{e<e0}
Let $e_0$ be the first characteristic exponent of $f$ at $\m$. Let $e=e(x,y,f)$. 

The following conditions are equivalent:
\begin{enumerate}
\item $e<e_0$,
\item $e$ is an integer, and there exists a regular system of parameters of the form $(y-ux^e,x)$, with $u$ a unit of $R$, satisfying $e(x,y,f)<e(x,y-ux^e,f)$.
\item 
$e$ is an integer and $\ini_{\nu_{x,y,e}}f$ is the $\mu$-th power of a linear form in $\ini_{\nu_{x,y,e}}y$ and $\ini_{\nu_{x,y,e}}x^e$; more specifically, there exist $c,d\in k$ such that
\begin{equation}
\ini_{\nu_{x,y,e}}f=(\ini_{\nu_{x,y,e}}y-c\ \ini_{\nu_{x,y,e}}x^e)^\mu.\label{eq:linearmu}
\end{equation}
\end{enumerate}
\end{proposition}
\begin{proof}
1) $\Longrightarrow$ 2) Since $e_0>e$, there exists a change of coordinates
\begin{eqnarray}
x_1&=&a_1x+a_2y^{l_1},\\
y_1&=&b_1y+b_2x^{l_2}
\end{eqnarray}
with $a_1, a_2, b_1, b_2 \in R^{*}$ such that
\begin{equation}
e_1:=e(x_1,y_1,f)>e(x,y,f).\label{eq:e1}
\end{equation}
Replacing $x$ by $a_1^{-1}(x_1-a_2y^{l_1})$ does not change $e$, therefore, we may assume that $x_1=x$. Since $b_1\in R^{*}$,  we may assume that $b_1=1$. Now we will prove that $e=\l_2$.\\

Write
\begin{equation}
f=\sum\limits_{i+je_1\ge\mu e_1}a_{ij}x^{i}y_1^{j}=
\sum\limits_{i+je_1\ge\mu e_1}a_{ij}x^{i}\left(y+b_2x^{l_2}\right)^{j},\label{eq:substitution}
\end{equation}
where $a_{0\mu}\in R^*$. Consider a monomial of the form $a_{ij}x^iy_{1}^j$ with 
\begin{equation}
i+je_1\geq \mu e_1.\label{eq:geqmu1}
\end{equation}
The element $x^iy_{1}^j$ belongs to the monomial ideal of $R$ generated by the set
\[
\left\{\left.x^{i+sl_2}y^{j-s}\ \right|\ s\in\{0,\dots,j\}\right\}.
\]
Let
\begin{equation}
e'=\min\{e_1,l_2\}.\label{eq:e'}
\end{equation}
Let us prove that $e=e'$. Indeed, if $i,j$ satislfy (\ref{eq:geqmu1})  then
\begin{equation}
i\ge(\mu-j)e_1.\label{eq:mu-j}
\end{equation}
If $s\in\{0,\dots,j\}$ then, since $l_2\ge e'$, we obtain
\begin{equation}
i+sl_2\ge(\mu-j)e_1+sl_2\ge(\mu-(j-s))e'.\label{eq:+sl_2}
\end{equation}
Thus $e'\le e$. Combined with (\ref{eq:e1}) and (\ref{eq:e'}), this proves that
\begin{equation}
e'=l_2<e_1.\label{eq:e'=l2}
\end{equation}
Combining (\ref{eq:e'=l2}) with (\ref{eq:+sl_2}), we obtain
\begin{equation}
i+sl_2\ge(\mu-(j-s))l_2\label{eq:mu-je}
\end{equation}
and the inequality is strict unless $i=0$ and $j=\mu$. Thus
\[
e=\min\left\{\left.\frac\alpha{\mu-\beta}\ \right|\ x^\alpha
y^\beta\in I(x,y,f),\beta<\mu\right\}=\min\left\{\left.\frac{sl_2}{\mu-(\mu-s)}\ \right|\ s\in\{1,\dots\mu\}\right\}=l_2.
\]
Therefore, $e\in \N$ and $y_1=y+ux^e$ with $u\in R^{*}$ satisfies the conclusion of (2).\\

2) $\Longrightarrow$ 3) Let $e'=e(x,y-ux^e,f)$ and write
$f=(y-ux^e)^{\mu}+\sum\limits_{\begin{array}{l}i+je'\geq\mu e'\\(i,j)\ne(0,\mu)\end{array}}a_{ij}x^{i}(y-ux^e)^{j}$. To prove 3) it is sufficient to prove that
$\nu_{x,y,e}\left(\sum\limits_{\begin{array}{l}i+je'\geq\mu e'\\(i,j)\ne(0,\mu)\end{array}}a_{ij}x^{i}(y-ux^e)^{j}\right)>\mu e$. Now, $\sum\limits_{\begin{array}{l}i+je'\geq\mu e'\\(i,j)\ne(0,\mu)\end{array}}a_{ij}x^{i}(y-ux^e)^{j}$ is contained in the monomial ideal (with respect to $(x,y)$), generated by monomials of the form ${{j}\choose{s}}a_{ij}x^{i}y^{j-s}(ux^e)^{s}$ with $0\leq s\leq j$, $i+e'j\geq \mu e'$ and if $j=\mu$ then $i>0$. Now we have to prove that the quantity $q=i+se+(j-s)e$ is strictly greater than $\mu e$. We have $q=i+je$. If $j= \mu$, then $i>0$ and $q>\mu
e$. If $j>\mu$ then $q>\mu e$. If $j<\mu$, then $q=i+je\geq e'(\mu-j)+je=\mu e-\mu e+e'(\mu-j)+je=\mu e+(\mu-j)(e'-e)>\mu
e$. This completes the proof of (3).\\

3) $\Longrightarrow$ 1) Choose $u\in R^{*}$ such that the natural image of $u$ in $k$ is $c$. We have
$$
f=(y-ux^e)^{\mu}+\sum\limits_{i,j} a_{ij}x^{i}y^{j},
$$
with $\nu_{x,y,e}\left(\sum a_{ij}x^{i}y^{j}\right)>\mu e$,  that is $i+je>\mu e$ for all the $(i,j)$ appearing in the sum.\\
Put $y_1=y-ux^e$. We will prove that $e'=e(x,y_1,f)>e$.\\

We have $f=y_1^{\mu}+\sum\limits_{i,j}\sum\limits_{s=0}^{j}{{j}\choose{s}}a_{ij}x^{i+es}y^{j-s}$ with $i+je>\mu e$ for each $(i,j)$ in the sum.\\

Now we have $e'\geq \frac{i+es}{\mu-(j-s)}>\frac{(\mu-j)e+es}{\mu-(j-s)}=e$ whenever $(j-s)<\mu$.

\end{proof}

\begin{remark} Let $e$ denote the first characteristic exponent of $f$. If $R$ is quasi-excellent, $f$ is reduced and $\mu\ge2$, we have
\begin{equation}
e\in\frac1{\mu!}\mathbb N,
\end{equation}
that is, $1\le e<\infty$. Since in this paper we work with arbitrary regular two-dimensional local rings and not just the
quasi-excellent ones, we will not use this fact in the sequel.
\end{remark}
Fix a simple sequence of point blowings up as in (\ref{eq:sequence}). Let $\mu_i(f)$ and $e_i(f)$ denote, respectively, the multiplicity and the first characteristic exponent of the strict transform of $f$ in $R_i$.
\begin{lemma}\label{decreasing_order} At least one of the following conditions holds:

(1) $(\mu_{i+1},e_{i+1})<_{lex}(\mu_i,e_i)$

(2) $e_i(f)=\infty$.
\end{lemma}
\begin{proof} To simplify the notation, we will consider the case when $i=0$, so that $R_0=R$. Assume that $e_i(f)\neq\infty$. Let $f_1$ be the strict transform of $f$ in $R_1$. We will follow the notation of (\cite{Z}, Appendix 5, pagse 365--367). Namely, let
$\bar g$ denote the directional form of the local blowing up $\pi_1$ and $\bar f$ the natural image of $f$ in $\gr_{\mathfrak m}R$.\\

Let $\mu=\mu_0$ denote the multiplicity of $f$ at $\mathfrak m$.

Since $\bar f$ is a homogeneous polynomial of degree $\mu$, the greatest power of $\bar g$ that could divide $\bar f$ is $\bar
g^\mu$. If $\bar{g}^\mu$ does not divide $ \bar{f}$, then by (\cite{Z}, Appendix 5, page 367, Proposition 2), we have
$\mu_1<\mu$  and (1) of the Lemma holds.

Assume that $\bar g^\mu\ \left|\ \bar f\right.$. Then $\deg\bar g=1$ and there exists a regular system of parameters $(x,y)$ such that $\bar g= \bar{y}$ and
$$
f=y^{\mu}+\sum\limits_{i+j>\mu}a_{ij}x^{i}y^{j},
$$
where the $a_{ij}$ are units of $R$. Let $e=e_0$ denote the first characteristic exponent of $f$ and choose $(x,y)$ in such a way that $e=e(x,y,f)$. Write
$f=y^{\mu}+\sum\limits_{\begin{array}{l}i+je=\mu e\\(i,j)\ne(0,\mu)\end{array}}a_{ij}x^{i}y^{j}+\sum\limits_{i+je>\mu e}a_{ij}x^{i}y^{j}$. Since $e<\infty$, there exists $(i,j)\neq (0,\mu)$, with $i+je=\mu e$. 

Now
$f_1=y_1^{\mu}+\sum\limits_{\begin{array}{l}i+je=\mu e\\(i,j)\ne(0,\mu)\end{array}}a_{ij}x_1^{i+j-\mu}y_1^{j}+
\sum\limits_{i+je>\mu e}a_{ij}x_1^{i+j-\mu}y_1^{j}$.\\

Note that for each $(i,j)$ with $i+je=\mu e$ we have $(i+j-\mu)+j(e-1)=\mu(e-1)$ and for each $(i,j)$ with $i+je>\mu e$ we have
$(i+j-\mu)+j(e-1)>\mu(e-1)$.\\

If $s>e-1$ then for $(i,j)$ with $i+je=\mu e$ we have $(i+j-\mu)=(\mu-j)(e-1)<(\mu-j)s$, hence $(i+j-\mu)+js<\mu s$.

If $e-1<1$ then $\mu_1<\mu$. Otherwise, if $e-1\ge1$, the above considerations prove that $e(x_1,y_1,f_1)=e-1$.\\

By Proposition \ref{e<e0}, $\ini_{\nu_{x,y,e}}f$ is not a $\mu$-th power of a linear form in $\ini_{\nu_{x,y,e}}y$ and
$\ini_{\nu_{x,y,e}}x^e$. Hence $\ini_{\nu_{x_1,y_1,e-1}}f_1$ is not a $\mu$-th power of a linear form in
$\ini_{\nu_{x_1,y_1,e-1}}y_1$ and $\ini_{\nu_{x_1,y_1,e-1}}x_1^e$. Therefore $e_1=e(x_1,y_1,f_1)=e-1<e$. In all the cases (1) of the Lemma holds.
\end{proof}

\begin{lemma}\label{inf_blowing} Let $\pi^{*}$ be an infinite sequence of local blowings up belonging to $\Pi(R)$. Write $\pi^{*}$ as in  (\ref{eq:sequence}). Take an element $f\in R\setminus\{0\}$.

\noi(1) If $f$ is a unit in $R_{j_0}$ for some $j_0$, then $f$ is a unit in $R_j$ for all $j\in\mathbb N$.

\noi(2) At least one of the following conditions holds:

(a) there exists $i\in \N$ such that $f=x_{i}^{s}y_{i}^{t}u$ where $x_i$ and $y_i$ are regular parameters of $R_i$, $s$ and $t$ are natural numbers and $u$ is a unit of $R_i$

(b) there exists $i_0\in\mathbb N$ such that $e_i(f)=\infty$ for all $i\ge i_0$.

\end{lemma}
\begin{proof}
\begin{enumerate}
\item This follows directly from the fact that for all $j\in \N$, we have $\m_{j}=R_{j}\cap\m_{j+1}$.
\item First note that if $x_i$ and $y_i$ are regular parameters of $R_i$, then either $x_i=x_{i+1}y_{i+1}$ and $y_i=y_{i+1}$ or $x_i=x_{i+1}$ and $y_i=\frac{y_i}{x_i} x_{i+1}$, with $\frac{y_i}{x_i}$ is either a unit in $R_{i+1}$ or equal to $y_{i+1}$. 

Assume that condition $(b)$ does not hold. From Lemma \ref{decreasing_order} we deduce that for each $j'\in\N$ there exists $j>j'$ with $\mu_{j}<\mu_{j'}$. Hence there exists $j\in \N$ with $f_j\in R^{*}$. Now by definition of $f_j$, $f_{j-1}=x^{\mu_{j-1}}f_j$, therefore, using the paragraph above and induction, we get the result.

\end{enumerate}
\end{proof}

For an element $\pi^{*}\in\Pi(R)$ we denote $\bar{R}=\bigcup\limits_{i} R_i$. The ring $\bar{R}$ is an integral domain with quotient field $F$, dominating $R$ and $R_i$ for each $i\in \N$.
 
\begin{proposition}\label{singular}
Assume that there exists $f\in R$, satisfying condition (2)(b) in Lemma \ref{inf_blowing}. Then there exists a unique Krull valuation $\nu$ on $F$ such that $R_{\nu}$ dominates $\bar{R}$.\\ 

Moreover, we have:

\begin{enumerate}
\item $\rk\ \nu=2$.
\item The set of elements in $R$ satisfying condition (2)(b) in Lemma \ref{inf_blowing} is a prime ideal generated by an irreducible element $g$.
\item let $g_i$  denote the strict transform of $g$ in $R_i$. Then 
$\nu$ is the composition of the $g$-adic valuation of $F$ with the unique rank one Krull valuation, centered in the one-dimensional local rings $\frac{R_i}{(f_i)}$ for each $i\in\mathbb N$.
\end{enumerate}
\end{proposition}
\begin{proof}

Since $\bar{R}$ is an integral domain with field of fractions $F$, there exists a valuation ring $R_{\nu}$ dominating $\bar{R}$.Let $\Gamma$ denote the value group of this valuation. Now to prove the uniqueness of $\nu$ it is sufficient to prove the conditions (1), (2) and (3).

\begin{enumerate}
\item Let $f_i$ denote the strict transform of $f$ in $R_i$. We have $\mu_i=\mu_{i_0}>0$ for all $i\geq i_0$. Let $(x_{i_0},y_{i_0})$ be a regular system of parameters of $R_{i_0}$ such that $f_{i_0}=y_{i_0}^{\mu_{i_0}}+$ terms of higher degree. Since
$\mu_i=\mu_{i_0}$, we have $x_{i_0}=x_{i_0+1}$. Now $f_{i_0+1}=y_{i_0}^{\mu_{i_0}}+$ terms of higher or equal degree. Repeating the same reasoning, we see that for each $i\geq i_0$ we have $x_i=x_{i_0}$ and $f_{i}=x_{i+1}^{\mu_{i+1}}f_{i+1}=x_{i_0}^{\mu_{i_0}}f_{i+1}$. Thus $f_{i_0}=x_{i_0}^{(i-i_0)\mu_{i_0}}f_{i}$ for each $i>i_0$. Since $f_i\in R_{\nu}$ for all $i$, we have $\nu(f_i)>0$ for all $i$. Hence $\nu(f_{i_0})>i\mu_{i_0}\nu(x_{i_0})$ for all $i$, so $\nu(f_{i_0})$ cannot belong to a subgroup of $\Gamma$ of  rank $1$. Therefore $rk(\nu)=2$.\\ 

\item Let $\Gamma_1$ denote the isolated subgroup of $\Gamma$ of rank $1$ (that is, the unique proper non-trivial subgroup of
$\Gamma$). Let $P'$ be the prime ideal of $R_{\nu}$ associated to $\Gamma_1$. Let $P=P'\cap R$. Then $P$ is a prime ideal of height $1$ in $R$, therefore it is generated by an irreducible element $g$. Now $f=hg^n$ with $h\notin P$, hence $\nu(h)\in \Gamma_1$. Therefore, by the proof of (1), there exists $i$ such that $h_i$ the strict transform of $h$ in $R_i$ is a unit. Now the strict transform of $f$ in $R_i$ is $h_i.g_i^{n}$. Therefore $g$ must also satisfy condition (2)(b) in Lemma \ref{inf_blowing}. An element of $R$ satisfies condition (2)(b) in Lemma \ref{inf_blowing} if and only if it belongs to $P$.

\item This is a direct consequence of (1) and (2).
\end{enumerate}
\end{proof}

\begin{proposition}\label{val_ring_blowing}
Let $\pi^{*}$ be an element of $\Pi(R)$ and write $\pi^{*}$ as in (\ref{eq:sequence}). Suppose that $R$ does not contain an element $f$ satisfying condition (2) (b) of Lemma \ref{inf_blowing}. The following statements hold.

(1) The ring $\bar{R}$ is a valuation ring  with field of fractions $F$, dominating $R$ and $R_i$ for each $i\in\mathbb N$.

(2) Conversely, let $R_\mu$ be a valuation ring with field of fractions $F$, dominating $R$ and $R_i$ for each $i\in\mathbb N$. Then $R_\mu=\bar R$.
\medskip

In other words, the simple blowing up sequence $\pi^*$ and the valuation $\mu$ determine each other uniquely; they are equivalent sets of data.
\end{proposition}

\begin{proof} (1) Since $R<R_i<R_j$ for all natural numbers $i\le j$, $\bar{R}$ is a domain with quotient field $F$, dominating $R$ and $R_i$ for each $i\in\mathbb N$.
\medskip
First, consider the case when the sequence
$$
\pi^{*}:(R,\m)\overset{\pi_1}\longrightarrow(R_1,\m_1)\overset{\pi_2}\longrightarrow\dots\overset{\pi_n}\longrightarrow
(R_n,\m_n)
$$
is finite. Then $\bar{R}=R_n$. By definition, $\m_n$ is principal and $(R_n,\m_n)$ is a discrete valuation ring.\\

Next, assume that $\pi^{*}$ is infinite. To prove that $\bar{R}$ is a valuation ring, consider an element $f\in F^*$, and write
$f=\frac{f_1}{f_2}$ where $f_1$, $f_2\in R\setminus\{0\}$.\\

By Lemma \ref{inf_blowing} there exists $i\in \N$ such that $f_1=x_i^{s_1}y_i^{t_1}u_1$ and $f_2=x_i^{s_2}y_i^{t_2}u_2$, where $x_i$ and $y_i$ are local parameters in $R_i$, $s_1$, $t_1$, $s_2$ and $t_2$ are natural numbers and $u_1$, $u_2$ are units in $R_i$. Hence
\begin{equation}
f=x_i^sy_i^tu,\label{eq:fmonomial}
\end{equation}
where $s$ and $t$ are integers (not necessarily positive) and $u$ is a unit in $R_i$. If both $s$ and $t$ are non-negative then $f\in R_i\subset\bar R$, as desired.  If both $s$ and $t$ are non-postive then $\frac 1f\in R_i\subset\bar R$, as desired.  Otherwise assume, without loss of generality, that $s>0$ and $t<0$. Now after another blowing up, we have the following three possibilities:
 \begin{eqnarray}
 f&=&x_{i+1}^sy_{i+1}^{s+t} u\quad\text{ or}\label{eq:ss+t}\\
 f&=&x_{i+1}^{s+t}y_{i+1}^t u \quad\text{ or}\label{eq:s+tt}\\
 f&=& y_{i+1}^{s+t} v\label{eq:onlyoneparamter}
 \end{eqnarray}
 where $x_{i+1}$ and $y_{i+1}$ are local parameters in $R_{i+1}$ and, in the last equation, $v$ is a unit in $R_{i+1}$. If (\ref{eq:onlyoneparamter}) holds and $s+t\geq 0$ then $f\in R_{i+1}\subset \bar{R}$. If (\ref{eq:onlyoneparamter}) holds and $s+t\le0$ then $f^{-1}\in R_{i+1}\subset \bar{R}$. According to Remark \ref{characterizingvaluaitonrings} (2), $\bar R$ is a valuation ring.
 Finally, if either (\ref{eq:ss+t}) or (\ref{eq:s+tt}) holds, we notice that the blowing up $\pi_{i+1}$ has strictly decreased the quantity $|s|+|t|$. Since this quantity cannot decrease indefinitely, after finitely many steps we will arrive either at (\ref{eq:fmonomial}) with 
 $s$ and $t$ of the same sign or at (\ref{eq:onlyoneparamter}), thus reducing the problem to one of the previous cases. Note also that if $f$ is of the form (\ref{eq:fmonomial}) with $s$ and $t$ of the same sign then the blowing up $\pi_{i+1}$ brings $f$ to the form (\ref{eq:onlyoneparamter}). This completes the proof of (1).
\medskip

(2) Conversely, let $R_\mu$ be a valuation ring such that $R_i<R_{\mu}$ for all $i\in\mathbb N$. Taking the direct limit as $i$ tends to infinity, we obtain $\bar R<R_\mu$. Now part (2) follows from (1) and the implication (a)$\Longrightarrow$(c) of Remark \ref{characterizingvaluaitonrings} (2). However, we give below a direct proof of (2) for the sake of completeness.

If the sequence $\pi^*$ is finite, its last ring $R_n$ is a discrete valuation ring. Let $x_n$ be a local parameter of $R_n$. We have
$\mu(x_n)>0$ since $\mu$ is centered at $\m_n$. Now any element $f$ of $F^*$ can be written as $f=x_{n}^{s}u$ where $s\in\Z$ and $u$ is a unit of $R_n$ (hence also a unit of $R_{\mu}$) and therefore $f\in R_n$ if and only if $f\in R_{\mu}$.\\ 

If $\pi^{*}$ is infinite, let $f\in F^*$. As in the proof of part (1), there exists $i$ such that $f=x_i^{s}u$ or $f=y_{i}^{s}u$ with $x_i$ and $y_i$ local parameters in $R_i$, $s\in \Z$ and $u$ is a unit in $R_i$. Now since $\pi^{*}$ is infinite, we have $\nu(x_i)>0$ and $\nu(y_i)>0$ and $u$ is also a unit of $R_{\mu}$ since $R_{i}<R_{\mu}$. Hence to say that $f\in \bar{R}$ is equivalent to saying that $s\geq 0$, which is equivalent to saying that $\mu(f)\geq 0$. Therefore
$\bar{R}=R_{\mu}$.
\end{proof}

\begin{corollary}\label{val_blow_onetoone} The set of valuations of $F$ centered at $R$ is in a natural one-to-one correspondence with $\Pi(R)$.
\end{corollary}

\begin{proof}

For the sake of completeness, we now give an explicit description of the element of $\Pi(R)$, associated to a given valuation $\mu$ by the above bijection and vice versa.

Let $\mu$ be a valuation centered at $R$. The center of $\mu$ in $R$ is $\xi_{0}:=\m$. Consider the point blowing up
$\pi^{*}_{1}:X_1\longrightarrow\ Spec\ R$ along $\xi_0$. The center of $\mu$ in $X_1$ is the unique point $\xi_1\in X_1$ whose local ring $\mathcal{O}_{X_1,\xi_1}$ is dominated by $R_{\mu}$. Put $R_1:=\mathcal{O}_{X_1,\xi_1}$, and let
$\m_1:=\m_{X_1,\xi_1}$ be its maximal ideal. If $\m_1$ is principal, stop here. Otherwise, fix $x_1,y_1\in R_1$ such that
$\m_1=(x_1,y_1)$. We have $(R,\m)<(R_1,\m_1)<(R_{\mu},\m_{\mu})$. Now repeat the same procedure with
$(R,\m)$ replaced by $(R_1,\m_1)$. Continuing in this way we obtain the simple sequence $\pi^{*}(\mu)$ (finite or infinite) of local point blowings up of $Spec\ R$.\\

Now we have two cases: \\

\textbf{Case 1}: The ring $R$ does not contain an element $f$ satisfying condition (2) (b) in Lemma (\ref{inf_blowing}).

Letting $\bar R=\lim\limits_{i\to\infty}R_i$, we have $\bar{R}=R_{\mu}$.\\

Conversely, take any element $\pi^{*}\in \Pi(R)$ and let $\bar{R}$ be as in Proposition \ref{val_ring_blowing}. Let $\mu$ to be the valuation on $F$  with valuation ring $\bar{R}$. It is clear that $\pi^{*}(\mu)$ described above is equal to $\pi^{*}$.\\

\textbf{Case 2}:  The ring $R$ does contain an element $f$ satisfying condition (2) (b) in Lemma (\ref{inf_blowing}) then by Proposition (\ref{singular})  the valuation
$\nu$ is uniquely determined by $\pi^{*}$.

\end{proof}

Recall that $k$ denotes the residue field of $R$.
\begin{theorem}\label{bijection}
There is a natural order preserving bijection  between valuations of $F$ centered at $R$ and valuations of $k(x,y)$ centered at $k[x,y]_{(x,y)}$. 
\end{theorem} 
\begin{proof}

By Proposition \ref{val_ring_blowing}, the set of valuations of $F$ centered at $R$ is in a natural one-to-one correspondence with $\Pi(R)$. Also by Proposition \ref{val_ring_blowing} applied to $k[x,y]_{(x,y)}$, the set of valuations of $k(x,y)$ centered at $k[x,y]_{(x,y)}$ is in a natural one-to-one correspondence with $\Pi(k[x,y]_{(x,y)})$. Finally, there is a natural one-to-one correspondence between $\Pi(R)$ and  $\Pi(k[x,y]_{(x,y)})$. Clearly all those correspondences fit together to give a natural order preserving bijection between valuations of $F$ centered on $R$ and valuations of $k(x,y)$ centered on $k[x,y]_{(x,y)}$. 
\end{proof}

\section{A Complete Set of Key Polynomials}\label{com_key_pol}

Let $\tilde{\nu}\in \V$. Fix local coordinates $x$ and $y$ such that $\nu(x)=1$. Let $K=k(x)$. Let $\nu$ be the valuation of $k(x,y)$ corresponding to $\tilde{\nu}$ under the bijection of Theorem \ref{bijection}.\\

The goal of this section is to construct a set of polynomials, complete for $\nu$ (the definition is given below). This set will be our main tool for constructing the valuative tree.\\

\subsection{Definition and Basic Properties of Key Polynomials}

For each strictly positive integer $b$, we write $\partial_{b}:=\frac{\partial^{b}}{b!\partial y^{b}}$, the $b$-th formal derivative
with respect to $y$.\\

For each polynomial $P\in K[y]$, let
$\epsilon_{\nu}(P):=\max\limits _{b\in\mathbb{N}}\left\{ \frac{\nu(P)-\nu(\partial_{b}P)}{b}\right\}$,
\[
I_{\nu}(P):=\left\{ b\in\mathbb{N}\ \left|\ \frac{\nu(P)-\nu(\partial_{b}P)}{b}=\epsilon_{\nu}(P)\right.\right\} .
\]
and
$b_{\nu}(P):=\min I_{\nu}(P)$.
\begin{definition}
Let $Q$ be a monic polynomial in $K[y]$, with $\nu(Q)\geq \nu(y)$. We say that $Q$ is an \textbf{abstract key polynomial} for
$\nu$ if for each polynomial $f$ satisfying
$$
\epsilon_{\nu}(f)\geq\epsilon_{\nu}(Q),
$$
we have $\deg(f)\geq\deg(Q)$. 
\end{definition}

For the rest of the paper, we will say \textbf{key polynomial} for \textbf{abstract key polynomial}.

For a monic polynomial $Q$ in $K[y]$ and a $g\in K[y]$ we can write $g$ in a unique way as 
\begin{equation}
g=\sum\limits _{j=0}^{s}g_{j}Q^{j},\label{eq:Qexpansion}
\end{equation}
with all the $g_{j}\in K[y]$ of degree strictly less than $\deg(Q)$.
\begin{definition}
For every monic polynomial $Q$ and every polynomial $g$ in $K[y]$, we call the expression (\ref{eq:Qexpansion}) the
$Q$\textbf{-expansion} of $g$. We define $\nu_{Q}(g):=\min\limits _{0\leq j\leq s}\nu(g_{j}Q^{j})$. We call $\nu_{Q}$ the \textbf{truncation} of $\nu$ with respect to $Q$.
\end{definition}
\begin{proposition}\textbf{(Proposition 12 of \cite{DSM})} If $Q$ is a key polynomial for $\nu$ then $\nu_Q$ is a valuation.
\end{proposition}

The following proposition is a direct consequence of Proposition 19 \cite{DSM} that states that each key polynomial for $\nu$ is
$\nu$-irreducible.
\begin{proposition}\label{irreducible} If $Q$ is a key polynomial for $\nu$ then $Q$ is irreducible.
\end{proposition}

\begin{proposition}\label{linear_kp} Every monic linear polynomial $Q$ in $K[y]$ is a key polynomial for $\nu$.
\end{proposition}
\begin{proof}
For any monic linear polynomial $Q\in K[y]$ and for any $c\in K$, we have 
$$\epsilon_{\nu}(Q)=\nu(Q)>-\infty=\epsilon_{\nu}(c).$$ 
\end{proof}

The first part of the next proposition is Theorem 27 of \cite{DSM} and the second part is obvious.\\

\begin{proposition}\label{truncations_value_groups} (1) Let $\mu$ be a valuation of $K(y)$ such that $\mu<\nu$, and let $Q$ be a monic polynomial of minimal degree in $y$ such that $\mu(Q)<\nu(Q)$. Then $Q$ is a key polynomial for $\nu$.\\
\smallskip

(2) Furthermore, we have $\mu< \nu_Q\leq \nu$ and the value group $\Gamma_Q$ of $\nu_Q$ is equal to $\Gamma_\mu+\beta \Z$ where $\Gamma_\mu$ is the value group of $\mu$ and $\beta=\nu(Q)$.
\end{proposition}

Let $\mu$ be a valuation of $K(y)$ such that $\mu<\nu$, and let $Q$ be a monic polynomial of minimal degree in $y$ such that $\mu(Q)<\nu(Q)$. Let $\beta\in\bar{\R}$, with $\mu(Q)<\beta$.\\

We define a new valuation $\mu'$ in the following way:

For a polynomial $g\in K[y]$, let $g=\sum\limits _{j=0}^{s}g_{j}Q^{j}$ be the $Q$-expansion of $g$. Put
$$
\mu'(g):=\min\limits _{0\leq j\leq s}\{\mu(g_{j})+j\beta\}.
$$
We call $\mu'$ the \textbf{augmented valuation} constructed from $\mu$, $Q$, and $\beta$, and we denote it by
$[\mu,\ Q,\ \beta]$.\\

For further details on augmented valuations, see \cite{V1}.

\subsection{A Complete Set of Key Polynomials: the Definition}

Let $\beta_0=\nu(x)=1$ and $\beta_1=\nu(y)$. Let $\Gamma_{\nu}=\nu(F)\subset\bar{\R}$ denote the augmented value group of $\nu$. \\

For an element $\beta\in\Gamma_{\nu}$, let $\mathbf{P}_{\beta}$ be as defined at thte end of \S2, but with $F$ replaced by
$K(y)$:
$$
\mathbf{P}_{\beta}=\{f\in K(y)\ |\ \nu(f)\ge\beta\}.
$$
\begin{definition}{\label{de9.1}} A {\bf complete set of key polynomials} for $\nu$ is a set
$$
\mathbf{Q}=\{Q_i\}_{i\in I}
$$
where $I$ is a well ordered set, each $Q_i$ is a key polynomial in $K[y]$ for $\nu$, and for each $\beta\in\g_{\nu}$ the additive group $\mathbf{P}_\beta\cap K[y]$ is generated by products of the form $a\prod\limits_{j=1}^sQ_{i_j}^{\gamma_j}$, $a\in K$, such that
$\sum\limits_{j=1}^s\gamma_j\nu(Q_{i_j})+\nu(a)\ge\beta$.
\end{definition}
In \cite{HOS} it is proved that every valuation $\nu$ admits a complete set $\mathbf{Q}=\{Q_i\}_{i\in I}$ of key polynomials. 

\begin{remark}
If $\mathbf{Q}=\{Q_i\}_{i\in I}$ is a complete set of key polynomials for $\nu$, we will always assume that the well ordering of $I$ has the following property: for $i<j$ in $I$, we have $\nu(Q_i)<\nu(Q_j)$.
\end{remark}

\subsection{Basic Structure}\label{basic_structure}

Let $\mu$ be a valuation of $K(y)$  with $\mu<\nu$. Suppose that the subset $\Gamma_{\mu+}$ of positive values of
$\Gamma_{\mu}=\mu(K(y))$ is a well ordered set (with the standard order relation in $\bar{\R}$). Note that this assumption is equivalent to saying that $\Gamma_\mu\cong\mathbb Z$.\\

We will use the following notation:

\begin{enumerate}
\item Let $d_{\mu}(\nu)$ be the minimal degree of a monic polynomial $f$ in $K[y]$ satisfying $\mu(f)<\nu(f)$.
\item Put $\Phi_{\mu}(\nu):=\left\{Q\in K[y]\mid Q\ is\ monic,\ \deg_{y}(Q)=d_{\mu}(\nu),\ \mu(Q)<\nu(Q)\right\}$.
\item Put $\Psi_{\mu}(\nu):=\nu\left(\Phi_{\mu}(\nu)\right)=\left\{\nu(Q)\in K[y]\mid Q\ is\ monic,\ \deg_{y}(Q)=d_{\mu}(\nu),\ \mu(Q)<\nu(Q)\right\}$.
\end{enumerate}

\begin{proposition}\label{le_me}
The set $\Psi_{\mu}(\nu)$ is contained in $\Gamma_{\mu+}$ or in $\Gamma_{\mu+}\cup \{\alpha\}$, where $\alpha$ is a maximal element of $\Psi_{\mu}(\nu)$ if it exists. Moreover, $\Psi_{\mu}(\nu)$ is bounded below by $d\mu(y)\geq 0$.
\end{proposition}
\begin{proof}
Let $d:=d_{\mu}(\nu)$. First we will show that $\Psi_{\mu}(\nu)$ is bounded below by $d\cdot\mu(y)$.\\ 
Let $\beta\in \Psi_{\mu}(\nu)$ and choose $Q\in \Phi_{\mu}(\nu)$ such that $\nu(Q)=\beta$. \\
Suppose that $\beta<d\mu(y)$ and write $Q=y^d+g$, with $g\in K[y]$, $\deg_{y}(g)<d$.\\
Since $\mu\left(y^d\right)=d\mu(y)>\beta=\nu(Q)>\mu(Q)$, we have $\mu(Q)=\mu(g)$.\\
Since $\nu(y^d)\geq\mu(y^d)=d\mu(y)>\beta=\nu(Q)$, we have $\nu(Q)=\nu(g)$,\\
but $\nu(g)=\mu(g)$ by definition of $d$, therefore $\nu(Q)=\mu(Q)$, which is a contradiction.\\

Now we will prove that any element $\beta\in\Psi_{\mu}(\nu)$ which is not a maximal element must be in $\Gamma_{\mu}$.\\

Suppose that $\beta$ and $\alpha$ are two elements of $\Psi_{\mu}(\nu)$ such that $\beta<\alpha$.\\
Choose $Q$ and $Q'$ in $\Phi_{\mu}(\nu)$ such that $\nu(Q)=\beta$ and $\nu(Q')=\alpha$.\\
Write $Q'=Q+z$ with $z\in K[y]$, $\deg_{y}(z)<d$.\\
Since $\alpha>\beta$, we have $\nu(Q)=\nu(z)$. But $\nu(z)=\mu(z)$ by definition of $d_{\mu}(\nu)$. Hence
$\beta\in\Gamma_{\mu}$.\\
 
\end{proof}

From Proposition \ref{le_me} we see that $\Psi_{\mu}(\nu)$ is well ordered.\\
We will denote by $\beta_{\mu}(\nu)$ the smallest element of $\Psi_{\mu}(\nu)$.\\

Choose $Q\in \Phi_{\mu}(\nu)$ such that $\nu(Q)=\beta_{\mu}(\nu)$. By Proposition \ref{truncations_value_groups}, $Q$ is a key polynomial for $\nu$, the truncation $\nu_Q$ is a valuation with $\mu<\nu_Q\leq \nu$, and the augmented value group
$\Gamma_Q$ of $\nu_Q$ is $\Gamma_{\mu}+\beta_{\mu}(\nu)\Z$. Hence the set $\Gamma_{Q+}$ of positive values of
$\Gamma_Q$ is a well ordered set.\\

If $\nu_Q<\nu$, we can repeat the same process as above with $\mu$ replaced by $\nu_Q$.\\

Moreover, the valuation $\nu_Q$ does not depend on the choice of $Q$, as we will prove in the following proposition.\\

\begin{proposition}\label{unique_val} With the notation as above, if $Q'$ is another polynomial in $\Phi_{\mu}(\nu)$ such that
$\nu(Q')=\beta_{\mu}(\nu)$ then $\nu_Q=\nu_{Q'}$.
\end{proposition}
\begin{proof}
Let $f$ be a polynomial of minimal degree such that $\nu_Q(f)\neq\nu_{Q'}(f)$ and suppose that $\nu_Q(f)<\nu_{Q'}(f)$.\\
Clearly $\deg_{y}(f)\geq d_{\mu}(\nu)$. Let $f=a_sQ^{s}+\dots+a_0$ be the $Q$-expansion of $f$ and let
$$
g=a_{s-1}Q^{s-1}+\dots+a_0.
$$
By definition of $\nu_Q$ we have $\nu(a_iQ^{i})\geq \nu_Q(f)$ for each $0\leq i\leq s$.\\

Suppose first that $\nu(a_sQ^{s})>\nu_Q(f)$. Then $\nu_Q(f)=\nu_Q(g)$.

Since $\deg_{y}(g)< \deg_{y}(f)$, we have $\nu_Q(g)=\nu_{Q'}(g)$. Therefore
$$
\nu_{Q'}(a_sQ^{s})=\nu(a_sQ^{s})>\nu_Q(g)=\nu_{Q'}(g).
$$
This implies that $\nu_{Q'}(f)=\nu_{Q'}(g)$, that leads to $\nu_{Q'}(f)=\nu_{Q}(f)$ which is a contradiction.\\
We have proved that $\nu_Q(f)=\nu(a_sQ^{s})$.\\

Write $Q'=Q+z$, with $\deg_{y}(z)<d_{\mu}(\nu)$. Then $f=a_s(Q'-z)^{s}+\dots+a_0$, and the $Q'$-expansion of $f$ involves $a_sQ'^{s}$. Therefore we have\\
$\nu(a_sQ'^{s})=\nu_{Q'}(a_sQ'^{s})\geq \nu_{Q'}(f)>\nu_{Q}(f)=\nu(a_sQ'^{s})$, which is a contradiction.
\end{proof}
\begin{lemma}\label{kp_euclidian}
Let $f\in K[y]$ be such that
\begin{equation}
\nu_Q(f)=\mu(f)\label{eq:nuQ=mu}
\end{equation}
and let $f=qQ+r$ be the Euclidean division of $f$ by $Q$ in $K[y]$. Then
$\nu_Q(f)=\nu_Q(r)<\nu_Q(qQ)$.  
\end{lemma}
\begin{proof} By definition of $\nu_Q$, we have
\begin{equation}
\nu_Q(qQ)\geq\nu_Q(f).\label{eq:qQgef}
\end{equation}
Suppose we have equality, aiming for contradiction:
\begin{equation}
\nu_Q(qQ)=\nu_Q(f).\label{eq:qQ=f}
\end{equation}
Then
\begin{equation}
\nu_Q(r)\ge\nu_Q(f).\label{eq:nuQrgenuQf}
\end{equation}
By definition of $\nu_Q$, we have
\begin{equation}
\mu(r)=\nu_Q(r).\label{eq:mur=nuQr}
\end{equation}
Combining (\ref{eq:nuQ=mu}), (\ref{eq:nuQrgenuQf}) and (\ref{eq:mur=nuQr}), we obtain $\mu(f)\le\mu(r)$. Then
\begin{equation}
\mu(f)\le\mu(qQ).\label{eq:muflemuqQ}
\end{equation}
We have
\begin{equation}
\nu_Q(qQ)>\mu(qQ).\label{eq:nuQ>mu}
\end{equation}
Combining (\ref{eq:nuQ=mu}), (\ref{eq:qQ=f}), (\ref{eq:muflemuqQ}) and (\ref{eq:nuQ>mu}), we get
$\nu_{Q}(qQ)>\mu(qQ)\ge\mu(f)=\nu_Q(f)=\nu_{Q}(qQ)$, which is a contradiction.
\end{proof}
\begin{proposition}\label{Form_of_KP}(Theorem 9.4 \cite{M2},Theorem 1.11 \cite{V1}) Let $Q'$ be a monic polynomial of minimal degree among those satisfying
$$
\nu_Q(Q')<\nu(Q').
$$
Then the $Q$-expansion of $Q'$ is given by $Q'=Q^s+a_{s-1}Q^{s-1}\dots+a_0$ with
$$
\nu_Q(Q')=s\beta_{\mu}(\nu)=\nu(a_0).
$$
\end{proposition}
\begin{proof}
First, let $Q'=qQ+a_0$ be the Euclidean division of $Q'$ by $Q$. We have
$$
\nu(Q')>\nu_Q(Q')=\inf\{\nu_Q(qQ),\nu(a_0)\}
$$
by definition of $\nu_Q$. Hence $\nu(qQ)=\nu(a_0)=\nu_Q(Q')$.\\

Since $\deg_y(a_s)<d_{\mu}(\nu)$ and $Q$ is irreducible in $K[y]$ by Proposition \ref{irreducible}, there exist $g$ and $h$ in
$K[y]$ with $deg_y(g)<d_{\mu}(\nu)$ and $ga_s+hQ=1$. Now $ga_s=-hQ+1$, $\nu_Q(a_s)=\mu(a_s)$ and
$\nu_Q(g)=\mu(g)$, therefore by Lemma \ref{kp_euclidian} we have $\nu(ga_s)=\nu(1)<\nu(-hQ)$.\\

Now for each $j$, $0\leq j\leq s-1$, let $ga_j=q_{j}Q+r_j$ be the Euclidean division of $ga_j$ by $Q$ in $K[y]$. Since
$\nu_Q(a_j)=\mu(a_j)$ and $\nu_Q(g)=\mu(g)$, by Lemma \ref{kp_euclidian} we have $\nu(ga_j)=\nu(r_j)<\nu(q_jQ)$.\\

Consider the polynomial $Q''=Q^{s}+r_{s-1}Q^{s-1}+\dots+r_0$.\\
We have $Q''-gQ'=(r_s-ga_s)Q^{s}+(r_{s-1}-ga_{s-1})Q^{s-1}+\dots+r_{0}-ga_{0}$, with $r_s=1$.\\
Therefore
\[
\nu_Q(Q''-gQ')\geq\inf_{0\leq j\leq s}\{\nu(r_j-ga_j)+j\beta_{\mu}(\nu)\}>\inf_{0\leq j\leq s}\{\nu(r_j)+j\beta_{\mu}(\nu)\}=\nu_Q(Q'')
\]
and $\nu_Q(Q''-gQ')>\nu_Q(Q'')=\nu_Q(gQ')$.\\

If $\nu(Q'')=\nu_Q(Q'')$ then $\nu(Q''-gQ')\geq\nu_Q(Q''-gQ')>\nu_Q(Q'')=\nu(Q'')$ and we have
$\nu(Q'')=\nu(gQ')>\nu_Q(gQ')=\nu_Q(Q'')$, which is impossible. Hence $\nu(Q'')>\nu_Q(Q'')$.\\

Since $Q'$ is chosen of minimal degree we must have $\deg_y(Q'')\geq \deg_y(Q')$, but this implies that $\deg_{y}a_s=0$ and
$a_s=1$.\\

We still have to prove that $\nu_Q(Q')=\nu(Q^{s})$. By definition of $\nu_Q$ we have $\nu(Q^{s})\geq \nu_Q(Q')$.\\

Suppose that $\nu(Q^{s})>\nu_Q(Q')$, then $\nu_Q(Q')=\nu_Q(f)$, where $f=Q'-Q^{s}$.

But since $\deg_{y}(f)< \deg_{y}(Q')$, we have $\nu_Q(f)=\nu(f)$. We obtain
$$
\nu(Q^{s})=\nu_{Q}(Q^{s})>\nu_Q(Q')=\nu_Q(f)=\nu(f).
$$
This implies that $\nu(Q')=\nu(f)$, which leads to $\nu(Q')=\nu_{Q}(Q')$, a contradiction.
\end{proof}
\begin{corollary}\label{degreesandvalues}
If $Q'\in K[y]$ is monic with $\deg_y(Q')=d_{\mu}(\nu)$ and $\nu(Q')>\nu_Q(Q')$ then $\nu_Q(Q')=\beta_{\mu}(\nu)$. 
\end{corollary}
\begin{proof} This is a special case of Proposition \ref{Form_of_KP} when $\deg_y(Q')=d_{\mu}(\nu)$.
\end{proof}

\begin{proposition}\label{complete} Let $\{Q_i\}_{i\in I}$ be a set of key polynomials for $\nu$, with $I$  a well ordered set and
$\deg_y(Q_i)\leq \deg_y(Q_j)$ for $i<j$ in $I$. Let $\nu_i:=\nu_{Q_i}$ be the truncation associated with each $Q_i$. Then the set
$\{Q_i\}_{i\in I}$ is a complete set of key polynomials if and only if for each polynomial $f\in K[y]$ there exists an element $i\in I(\nu)$ such that $\nu(f)=\nu_{i}(f)$ and $\deg_yQ_i\le\deg_yf$.
\end{proposition}
\begin{proof} To say that $\{Q_i\}_{i\in I(\nu)}$ is a complete set of key polynomials is equivalent to saying that every $f\in K[x]$ can be written in the form
\begin{equation}
f=\sum\limits_\gamma a_\gamma\prod\limits_{j=1}^sQ_j^{\gamma_j},\label{eq:fingenmonomials}
\end{equation}
where $s$ is a strictly positive integer, $\gamma=\left(\gamma_1,\dots,\gamma_s\right)$ ranges over a finite subset of $\mathbb N^s$, $a_\gamma\in K$ and
$$
\sum\limits_{j=1}^s\gamma_j\nu(Q_j)+\nu(a)\ge\nu(f).
$$
\noindent\textit{Proof of ``if''.} Assume that for each $f\in K[x]$ there exists an integer $i$ as in the Proposition. We will construct the expression (\ref{eq:fingenmonomials}) recursively in $\deg_yf$. Assume that an expression of the form (\ref{eq:fingenmonomials}) exists for every polynomial of degree strictly less than $\deg_yf$. Put $\beta=\nu(f)$ and let $i$ be such that $\beta=\nu_{i}(f)$ and $\deg_yQ_i\le\deg_yf$. Write
\begin{equation}
f=\sum\limits_{j=0}^{s_i}c_{j}Q_{i}^{j},\label{eq:Qiexpansionoff}
\end{equation}
where each $c_{j}Q_i^{j}\in P_{\beta}$ and $\deg_yc_j<\deg_yQ_i\le\deg_yf$. By the induction assumption, each of the $c_j$ admits an expansion of the form (\ref{eq:fingenmonomials}). Substituting all of these expansions into (\ref{eq:Qiexpansionoff}), we obtain the desired expansion (\ref{eq:fingenmonomials}) of $f$. This completes the proof of ``if''.\\
\smallskip
    
\noindent\textit{Proof of ``only if''.} Conversely, take $f\in K[y]$. Let $\beta=\nu(f)$. 
     Write $f$ in the form (\ref{eq:fingenmonomials}).  Then
$$
\beta=\nu(f)\geq\nu_s(f)\geq \min \left\{\nu_s\left(c_\gamma\prod
Q_j^{\gamma_j}\right)\right\}=\min\left\{\nu\left(c_\gamma\prod Q_j^{\gamma_j}\right)\right\}\geq\beta.
$$
Thus all the inequalities in the above formula are equalities, so the natural number $i:=s$ satisfies the conclusion of ``only if''.
\end{proof}

\subsection{Construction of a Complete Set of Key Polynomials}\label{construction_complete}

First we put $Q_1:=y$ and $d_1(\nu)=1$. By Proposition \ref{linear_kp}, $Q_1$ is a key polynomial for $\nu$. Consider the valuation $\nu_1:=\nu_{Q_1}$. We have $\nu_1\leq \nu$. If $\nu_1=\nu$ then the algorithm stops here, 
we put $I(\nu)=\{1\}$ and $\{Q_i\}_{i\in I(\nu)}=\{Q_1\}$. We will prove in Proposition \ref{complete_deg} below that
$\{Q_i\}_{i\in I(\nu)}$ is complete for $\nu$.\\

Now suppose that $\nu_1<\nu$. Then we can apply the results of \S\ref{basic_structure} to $\mu=\nu_1$. Put
$d_{2}(\nu)=d_{\nu_1}(\nu)$, $\Phi_{1}(\nu)=\Phi_{\nu_1}(\nu)$, $\Psi_{1}(\nu)=\Psi_{\nu_1}(\nu)$ and
$\beta_2(\nu)=\beta_{\nu_1}(\nu)$.\\

Choose $Q_2\in\Phi_{1}(\nu)$ such that $\nu(Q_2)=\beta_2(\nu)$ and let $\nu_2:=\nu_{Q_2}$.\\

We have $\nu_1<\nu_2\leq\nu$. If $\nu_2=\nu$, then the algorithm stops here,
we put $I(\nu)=\{1,2\}$ and $\{Q_i\}_{i\in I(\nu)}=\{Q_1,Q_2\}$. By Proposition \ref{complete_deg} below, $\{Q_i\}_{i\in I(\nu)}$ is complete for $\nu$.\\

Otherwise, if $\nu_2<\nu$, we can apply the results of \S\ref{basic_structure} to $\mu=\nu_1$ and repeat the same process with
$\nu_1$ replaced by $\nu_2$.\\

Assume that for a certain natural number $n\ge2$ a set $\{Q_i\}_{i\geq n}$ has been constructed. If $\nu_n=\nu$ then, by Proposition \ref{complete} and Proposition \ref{complete_deg} below, $\{Q_i\}_{i\geq n}$ is complete for $\nu$. The construction stops here.

Otherwise, we have $\nu_n<\nu$. Let us apply the results of \S\ref{basic_structure} to $\mu=\nu_n$.  Put
$$
d_{n+1}(\nu)=d_{\nu_n}(\nu),
$$
$\Phi_n(\nu)=\Phi_{\nu_n}(\nu)$, $\Psi_n(\nu)=\Psi_{\nu_n}(\nu)$ and $\beta_{n+1}(\nu)=\beta_{\nu_n}(\nu)$.\\

Choose $Q_{n+1}\in\Phi_n(\nu)$ such that $\nu(Q_{n+1})=\beta_{n+1}(\nu)$ and let $\nu_{n+1}:=\nu_{Q_{n+1}}$.

Repeating this process, there are two possibilities. The first is that we find valuations $\{\nu_{i}\}_{i\leq n}$ and key polynomials
$\{Q_i\}_{i\leq n}$ such that $\nu_n=\nu$. The second is that we construct an infinite set $\{Q_{i}\}_{i\in \N}$ of key polynomials and valuations $\{\nu_{i}\}_{i\in \N}$. We will study this case after Proposition \ref{complete_deg}.\\

\begin{remark} Let $\{Q_{i}\}_{i\in I(\nu)}$ be the set constructed above, with $I(\nu)=\{1,2,\dots\}$ (possibly infinite). Even though the polynomials $\{Q_{i}\}_{i\in I(\nu)}$ are not uniquely determined, their degrees $\{d_{i}(\nu)\}_{i\in I(\nu)}$, their values
$\{\beta_{i}(\nu)\}_{i\in I(\nu)}$, the associated valuations $\{\nu_{i}\}_{i\in I(\nu)}$ are uniquely determined by
$\nu$, from the construction above and by Proposition \ref{unique_val}. As well, the sets $\Phi_{i}(\nu)$ and $\Psi_{i}(\nu)$, for
$i\in I(\nu)$, are uniquely determined by $\nu$ by construction. 
\end{remark}

\begin{proposition}\label{complete_deg} Let $\{Q_{i}\}_{i\in I(\nu)}$ be the set constructed above, with $I(\nu)$ possibly infinite. Let $f\in K[y]$ and suppose there exists $i_0\in I(\nu)$ such that  $\deg_y (f)<\deg_y (Q_{i_0})$. There exists $i<i_0$ such that $\deg_y(Q_i)\leq \deg_y(f)$ and $\nu_i(f)=\nu(f)$.
\end{proposition}
\begin{proof} Multiplying $f$ by a non-zero element of $K$ does not change the problem, therefore we may assume that $f$ is monic.\\

Now let $i$ be such that $(d_i(\nu),\beta_i(\nu))=max\{j\in I(\nu)\mid (d_j(\nu),\beta_j(\nu))\leq_{lex}(\deg_y(f),\nu(f))\}$. Now by construction, if $i+1$ does not exist, then $\nu(f)=\nu_i(f)$. Otherwise, if $i+1\in I(\nu)$, then
$(d_{i+1}(\nu),\beta_{i+1}(\nu))>(\deg_y(f),\nu(f))$, and by definition of $Q_{i+1}$ we have $\nu(f)=\nu_i(f)$.
\end{proof}

Suppose that $\{Q_{i}\}_{i\in I(\nu)}$ is constructed as above with $I(\nu)=\N$. Two cases are possible.

\begin{itemize}
\item[\textbf{Case 1}:] The set $\Phi_{i}(\nu)$ is finite for each $i\in I(\nu)$.
\item[\textbf{Case 2}:] There exists $i\in I(\nu)$ such that $\Phi_{i}(\nu)$ is infinite.
\end{itemize}

\begin{proposition}\label{maximalelement} If $\beta_{i+1}(\nu)$ is a maximal element of $\Psi_{i}(\nu)$ then
$\Psi_i(\nu)=\{\beta_{i+1}(\nu)\}$ and
\begin{equation}
\Psi_{i}(\nu)\cap\Psi_{i+1}(\nu)=\emptyset.\label{eq:emptyintersection}
\end{equation}
Moreover, if
$\Psi_{i+1}(\nu)\neq\emptyset$ then $d_{i+2}(\nu)>d_{i+1}(\nu)$ and $\beta_{i+2}(\nu)>\beta_{i+1}(\nu)$. If
$\beta_{i+1}(\nu)$ is not maximal in $\Psi_{i}(\nu)$, then $\Psi_{i+1}(\nu)=\Psi_{i}(\nu)\setminus\{\beta_{i+1}(\nu)\}$ and
$d_{i+2}(\nu)=d_{i+1}(\nu)$.
\end{proposition}

\begin{proof} First suppose that $\beta_{i+1}(\nu)$ is maximal in $\Psi_{i}(\nu)$. Since by definition $\beta_{i+1}(\nu)$ is the minimal element of $\Psi_{i}(\nu)$ and $\Psi_{i}(\nu)$ is totally ordered, we have $\Psi_{i}(\nu)=\{\beta_{i+1}\}$.\\

Now if $\Psi_{i+1}(\nu)=\emptyset$, the equality (\ref{eq:emptyintersection}) hods trivially. Thus we will assume that
$$
\Psi_{i+1}(\nu)\neq\emptyset.
$$ 
We will first prove that $d_{i+2}(\nu)>d_{i+1}(\nu)$.\\
By construction, we have $d_{i+2}(\nu)\geq d_{i+1}(\nu)$. Aiming for contradiction, suppose that we have equality. Take an element $Q\in\Phi_{i+1}(\nu)$.\\
We have $\deg_{y}(Q)=d_{i+2}(\nu)=d_{i+1}(\nu)$ and $\nu(Q)>\nu_{i+1}(Q)\geq \nu_{i}(Q)$, therefore $Q\in \Phi_{i}(\nu)$, by definition of $\Phi_{i}(\nu)$. Hence $\nu(Q)\in \Psi_{i}(\nu)=\{\beta_{i+1}(\nu)\}$.\\
On the other hand, $Q$ satifies the hypothses on $Q'$ in Corollary \ref{degreesandvalues}. Therefore
$$
\nu_{i+1}(Q)=\beta_{i+1}(\nu).
$$
This implies that $\nu(Q)=\nu_{i+1}(Q)$ which contradicts the fact that $Q\in\Phi_{i+1}(\nu)$.\\

We have proved that $d_{i+1}(\nu)<d_{i+2}(\nu)$. Now the fact that $\beta_{i+2}(\nu)>\beta_{i+1}(\nu)$ follows from Proposition \ref{Form_of_KP}, since $\beta_{i+2}(\nu)>\nu_{i+1}(Q_{i+2})=s\beta_{i+1}$, where $s=\deg_{Q_i}Q_{i+1}$.\\

Since $\beta_{i+2}(\nu)>\beta_{i+1}(\nu)$, we have $\Psi_{i}(\nu)\cap\Psi_{i+1}(\nu)=\emptyset$.\\

Next, suppose that $\beta_{i+1}(\nu)$ is not a maximal element in $\Psi_{i}(\nu)$. Let $\beta\in\Psi_{i}(\nu)\setminus\{\beta_{i+1}(\nu)\}$. Choose $Q\in \Phi_{i}(\nu)$ such that $\nu(Q)=\beta$ and write $Q=Q_{i+1}+z$. Since $\beta>\beta_{i+1}(\nu)$, we have $\nu(z)=\beta_{i+1}(\nu)$ and $\nu_{i+1}(Q)=\beta_{i+1}(\nu)<\beta=\nu(Q)$. We have proved that
$Q\in\Phi_{i+1}(\nu)$, therefore $d_{i+2}(\nu)=d_{i+1}(\nu)$ and $\beta\in \Psi_{i+1}(\nu)$. Thus
$\Psi_{i}(\nu)\setminus\{\beta_{i+1}(\nu)\}\subset\Psi_{i+1}(\nu)$.\\

Now let $\beta\in \Psi_{i+1}(\nu)$. By Corollary \ref{degreesandvalues}, we have $\beta>\beta_{i+1}(\nu)$. Take an element
$$
Q\in \Phi_{i+1}(\nu)
$$
such that $\nu(Q)=\beta$. We have $\deg_{y}(Q)=d_{i+2}(\nu)=d_{i+1}(\nu)$, and
$\nu(Q)>\nu_{i+1}(Q)\geq\nu_i(Q)$, hence $Q\in\Phi_{i}(\nu)$ and $\beta\in\Psi_{i}(\nu)\setminus\{\beta_{i+1}(\nu)\}$.

\end{proof}

\begin{corollary}\label{infinite_values} If $\Psi_{i}(\nu)$ is infinite for some $i\in I(\nu)$, then $\Psi_{i+1}(\nu)$ is infinite and
$$
\Psi_{i+1}(\nu)\subset\Psi_{i}(\nu).
$$
\end{corollary}

\begin{corollary}\label{discrete_group}
Let $i_0=inf\{i\in\N\mid \#\Psi_{i}(\nu)=\infty\}$, then for each $i\geq i_0$, the value group $\Gamma_i$ of $\nu_i$ is equal to the value group $\Gamma_{i_0}$ of $\nu_{i_0}$.
\end{corollary}
\begin{proof}
This is a direct consequence of Proposition \ref{le_me} (2) and Corollary \ref{infinite_values}.
\end{proof}

First, suppose we are in \textbf{Case 1}: 
\begin{proposition}\label{unbdegrees}
The degrees $d_{i}(\nu)$ are unbounded in $\N$.
\end{proposition}
\begin{proof} Take an element $i\in I(\nu)$. We will prove that there exists $j\in I(\nu)$ such that
$$
d_j(\nu)>d_{i+1}(\nu).
$$
Since $\Psi_{i}(\nu)$ is finite, it admits a maximal element $\alpha$. If $\beta_{i+1}(\nu)=\alpha$ then, by Proposition \ref{maximalelement}, we have $d_{i+2}(\nu)>d_{i+1}(\nu)$. Suppose that $\beta_{i+1}(\nu)$ is not maximal. By Proposition \ref{maximalelement} we have $\Psi_{i+1}(\nu)=\Psi_{i}(\nu)\setminus\{\beta_{i+1}(\nu)\}$, therefore $\alpha$ is also the maximal element of $\Psi_{i+1}(\nu)$. Now repeat the same reasoning: if $\beta_{i+2}(\nu)=\alpha$, then
$d_{i+3}(\nu)>d_{i+2}(\nu)=d_{i+1}(\nu)$, otherwise  we have $\Psi_{i+2}(\nu)=\Psi_{i+1}(\nu)\setminus\{\beta_{i+2}(\nu)\}$ and $\alpha$ is the maximal element of $\Psi_{i+2}(\nu)$. The process must end since $\Psi_{i}(\nu)$ is finite.
\end{proof}

\begin{theorem} The set of key polynomials $\{Q_i\}_{i\in I(\nu)}$ is complete for $\nu$.
\end{theorem}
\begin{proof}
Let $f\in K[y]$. By Proposition \ref{unbdegrees}, there exists $i\in I(\nu)$ such that $d_i(\nu)>\deg_{y}(f)$. Now the result follows from  Proposition \ref{complete} and Proposition \ref{complete_deg}.
\end{proof}

Suppose we are in \textbf{Case 2}.

\begin{theorem}
\begin{enumerate}
\item If $\nu^{-1}(\infty)=\{0\}$ then the set $\{Q_{i}\}_{i\in I(\nu)}$ is complete for $\nu$. 
\item If $\nu^{-1}(\infty)\neq \{0\}$ then there exists a key polynomial $Q_{\omega}$ for $\nu$, that generates the ideal $\nu^{-1}(\infty)$ and it is of minimal degree such that $\nu_{i}(Q_{\omega})<\nu(Q_{\omega})$ for all $i\in \N$. Moreover, the set $\{Q_{i}\}_{i\in I(\nu)}\cup \{Q_{\omega}\}$ is complete for $\nu$.
\end{enumerate}
\end{theorem}
\begin{proof} Fix a polynomial $f\in K[y]$ such that $\nu(f)<\infty$. We have the following inequalities: 
$$
\nu_1(f)\leq \dots\leq \nu_{i}(f)\leq \dots\leq \nu(f)<\infty.
$$
By Corollary \ref{discrete_group}, all those values belongs to the value group $\Gamma_{i_0}$ of $\nu_{i_0}$, where $i_0\in I(\nu)$ is defined in  Corollary \ref{discrete_group}. Now $\Gamma_{i_0}=\beta_0\Z+\beta_1\Z+\dots+\beta_{i_0}\Z$ is discrete, hence there exists a certain integer $j$ such that $\nu_{i}(f)=\nu(f)$ for any $i\geq j\in \N$.\\

Now 1) follows from Proposition \ref{complete}.\\

To prove 2), suppose that $\nu^{-1}(\infty)\neq \{0\}$. The set $\nu^{-1}(\infty)$ is an ideal in $K[y]$, it can be generated by one element. Choose $Q_{\omega}$ to be a monic polynomial that generates $\nu^{-1}(\infty)$. The polynomial $Q_\omega$ has minimal degree among the polynomials in $\nu^{-1}(\infty)$.\\

We have $\epsilon_{\nu}(Q_{\omega})=\infty$, and $Q_{\omega}$ of minimal degree with this property, hence $Q_{\omega}$ is a key polynomial for $\nu$. \\

We have $\nu=\nu_{\omega}:=\nu_{Q_{\omega}}$, and for any polynomial $f\in K[y]$, if $f\notin \nu^{-1}(\infty)$, by the discussion at the beginning of the proof,  there exists $i\in I(\nu)$ such that $\nu_{i}(f)=\nu(f)$, otherwise, if $f\in \nu^{-1}(\infty)$, then $\nu(f)=\nu_\omega(f)=\infty$. Hence the set $\{Q_{i}\}_{i\in\N}\cup \{Q_{\omega}\}$ is complete for $\nu$ Proposition \ref{complete}.
\end{proof}

For the rest of the paper if $Q_{\omega}$ exists, we put $I(\nu)=\N\cup \{\omega\}$.\\

We denote:
\begin{enumerate}
\item $D(\nu):=max_{i\in I(\nu)}\{d_i(\nu)\}$, if this maximum exists; otherwise, we put $D(\nu)=\infty$.
\item $N(\nu)$ the maximal element of $I(\nu)$, if this maximum exists, otherwise, we put $N(\nu)=\infty$.
\end{enumerate}

\begin{remark}\label{invariants} From the construction above we see that:
\begin{enumerate}
\item $N(\nu)=\infty$ if and only if $I(\nu)=\N$.
\item If $D(\nu)=\infty$ then $I(\nu)=\N$.
\item If $D(\nu)<\infty$ and ($N(\nu)=\infty$ or $N(\nu)=\omega$), we are in the case where there exists $i\in I(\nu)$ such that $\#\Psi_i(\nu)=\infty$.\\
\end{enumerate}
\end{remark}

\section{The order relation on $\V$}\label{order_val}

\subsection{Invariants of comparable valuations}

Let $\tilde{\mu}$ and $\tilde{\nu}$ be two elements of $\V$ with $\tilde{\mu}<\tilde{\nu}$. Choose local coordinates $x$ and $y$ such that $\tilde{\nu}(x)=\tilde{\mu}(x)=1$.\\ 

Put $K=k(x)$ and let $\mu$ and $\nu$ be the valuations of $K(y)$, corresponding to $\tilde\mu$ and $\tilde\nu$, respectively.\\

Let $\{Q_i\}_{i\in I(\nu)}$ be a complete set of key polynomials associated to $\nu$.

\begin{lemma}
There exists $i\in I(\nu)$ such that $\mu(Q_i)<\nu(Q_i)$.
\end{lemma}
\begin{proof}
Suppose that for all $i\in I(\nu)$ we have $\mu(Q_i)=\nu(Q_i)$.\\

Since $\nu>\mu$, there exists $f\in K[y]$ such that $\nu(f)>\mu(f)$. Choose $f\in K[y]$ of minimal degree among the polynomials having this property.\\

Since $\{Q_i\}_{i\in I(\nu)}$ is complete for $\nu$, there exists $i\in I(\nu)$ such that $\nu(f)=\nu_{i}(f)$.\\

Let $f=qQ_{i}+r$ be the Euclidean division of $f$ by $Q_i$.\\

We have $\nu(f)>\mu(f)\geq inf\{\mu(qQ_i),\mu(r)\}= inf\{\nu(qQ_i),\nu(r)\}\ge\nu_{i}(f)$, which is a contradiction. 
\end{proof}

Let $i_0:=min\left\{i\in I(\nu)\mid \mu(Q_i)<\nu(Q_i)\right\}$.

\begin{proposition}\label{first_case}
If $i_0=1$, then $I(\mu)=\{1\}$ and $\mu=\mu_1<\nu_1$.
\end{proposition}
\begin{proof}
Since $i_0=1$, we have $\mu(y)<\nu(y)$. It is sufficient to prove that $\mu=\mu_1$.\\

Suppose there exists $f\in K[y]$, with $\mu(f)>\mu_1(f)$. Choose $f$ of minimal degree satisfying $\mu(f)>\mu_1(f)$ and 
let $f=qy+r$ be the Euclidean division of $f$ by $y$. \\

We have $\mu(f)>\mu_1(f)=\inf\{\mu(qy),\mu(r)\}$, therefore, $\mu(f)>\mu(qy)=\mu(r)$.\\
Since $\nu(f)\geq\mu(f)$, we have $\nu(f)>\mu(r)=\nu(r)$. This implies that $\nu(f)>\nu(qy)=\nu(r)$.\\
Finally, we get $\nu(qy)=\nu(r)=\mu(r)=\mu(qy)$.
But $\nu(qy)=\nu(q)+\nu(y)>\mu(q)+\mu(y)=\mu(qy)$, and we have a contradiction.
\end{proof}

\begin{proposition}\label{equal_subvaluations}
If $i_0>1$ then for any $i\in I(\nu)$ with $i<i_0$, we have $i\in I(\mu)$, $\mu_i=\mu_{Q_i}=\nu_i$, $\beta_{i}(\mu)=\beta_{i}(\nu)$ and $d_i(\mu)=d_{i}(\nu)$.
\end{proposition}
\begin{proof}
Since $i_0>1$, we have $\mu(y)=\nu(y)$, hence $\mu_1=\nu_1$, $\beta_{1}(\mu)=\beta_{1}(\nu)$ and $d_{1}(\mu)=d_{1}(\nu)$.\\

Take an integer $i$, $1<i<i_0$ (in particular $i\in \N$), and suppose inductively that for all $j$, $1\leq j<i$, we have $\nu_j=\mu_j$, $\beta_{j}(\mu)=\beta_{j}(\nu)$ and $d_{j}(\mu)=d_{j}(\nu)$.\\ 

We will first prove that $d_i(\mu)=d_i(\nu)$.

If $f$ is a monic polynomial with $\mu(f)>\mu_{i-1}(f)=\nu_{i-1}(f)$,
then $\nu(f)\geq \mu(f)>\nu_{i-1}(f)$, hence $d_{i}(\nu)\leq d_{i}(\mu)$.

To prove the equality, we will prove that $\mu(Q_i)>\mu_{i-1}(Q_i)$. Indeed, by definition of $i_0$ we have $\mu(Q_i)=\nu(Q_i)$ since $i<i_0$.\\
Hence $\mu(Q_i)=\nu(Q_i)>\nu_{i-1}(Q_i)=\mu_{i-1}(Q_i)$.\\
Therefore $d_{i}(\nu)= d_{i}(\mu)$.\\

Now to prove that $\beta_{i}(\mu)=\beta_{i}(\nu)$, we still have to prove that if $f$ is a monic polynomial with $\mu(f)>\mu_{i-1}(f)$ and $\deg_y(f)=d_{i}(\mu)$, then $\mu(f)\geq\mu(Q_i)$. In this case, we will have $\beta_{i}(\mu)=\mu(Q_i)$, and since by definition of $i_0$, $\mu(Q)=\nu(Q)$, we get the desired equality.\\
Let $f$ be such a polynomial. Write $f=Q_i+g$ with $\deg_y(g)<d_{i}(\mu)$.\\
If $\mu(f)<\mu(Q_i)$ then $\mu(f)=\mu(g)=\mu_{i-1}(g)=\nu_{i-1}(g)=\nu(g)$ and
$$
\nu(Q_i)\geq\mu(Q_i)>\mu(f)=\nu(g).
$$
Hence $\nu(g)=\nu(f)$ and $\nu(Q_i)>\nu(f)$.\\
We have proved that $\beta_i(\nu)>\nu(f)$ and $\nu(f)\geq\mu(f)>\mu_{i-1}(f)=\nu_{i-1}(f)$, which contradicts the definition of $\beta_{i}(\nu)$.\\

Since $Q_i$ is monic of degree $d_{i}(\mu)$ with $\mu(Q_{i})=\beta_{i}(\mu)$, we have $\mu_i=\mu_{Q_i}$. Since
$\beta_{i}(\mu)=\beta_{i}(\nu)$ then $\mu_i=\nu_i$.
\end{proof}

\begin{proposition} We have $i_0<\omega$. In other words, $i_0\in\N$.
\end{proposition}
\begin{proof}
Suppose that $\omega\in I(\nu)$ and $i_0=\omega$.\\

By Remark \ref{invariants} (3) and Corollary \ref{discrete_group}, there exists $i_1\in\N$ such that $\beta_{i}(\nu)\in \Gamma_{i_1}$ for all $i\geq i_1$, with $\Gamma_{i_1}=\Z+\beta_1\Z+\dots+\beta_{i_1}\Z\subset \R$. Hence for all $i\in \N$, we have
$\nu_i(Q_\omega)\in \Gamma_{i_1}$.

Let us show that
$$
\nu_1(Q_\omega)<\nu_2(Q_\omega)<\dots<\nu(Q_\omega)=\infty.
$$
Indeed, assume that there exists $i\in I(\nu)\setminus\{\omega\}$, such that $\nu_i(Q_\omega)=\nu_{i+1}(Q_{\omega})$, aiming for contradiction. Write $Q_{\omega}=qQ_{i+1}+r$ the Euclidean Division of $Q_{\omega}$ by $Q_{i+1}$. We have
$\nu_{i+1}(qQ_{i+1})>\nu_{i}(qQ_{i+1})\geq \inf\{\nu_{i}(Q_{\omega}), \nu_{i}(r)\}=
\inf\{\nu_{i+1}(Q_{\omega}), \nu_{i+1}(r)\}$, hence
$$
\nu_{i+1}(qQ_{i+1})>\nu_{i+1}(Q_{\omega}), \nu_{i+1}(r).
$$
This implies that $\nu(qQ_{i+1})>\nu(r)$. Hence $\nu(Q_{\omega})=\nu(r)$ and $\nu(r)=\infty$, then $r$ must be equal to $0$ and $q=1$, since $\nu^{-1}(\infty)=(\Q_{\omega})$. But $Q_{i+1}\neq Q_{\omega}$ and we have a contradiction.\\

By Proposition \ref{equal_subvaluations} we have $\nu_i=\mu_i$, hence
$$
\mu_1(Q_\omega)<\mu_2(Q_\omega)<\dots<\mu(Q_\omega).
$$
We have a strictly increasing sequence in $\Gamma_{i_1}$, it most be unbouded in $\R$, hence $\mu(Q_\omega)=\infty$. This contradicts the fact that $\mu(Q_\omega)<\nu(Q_\omega)$.
\end{proof}

\begin{proposition}
Either $\mu=\nu_{i_0-1}$ or $\mu=\mu_{i_0}<\nu_{i_0}$.
\end{proposition}
\begin{proof}
As in the proof of Proposition \ref{equal_subvaluations}, we have $d_{i_0}(\nu)\leq d_{i_0}(\mu)$.

Suppose first that $\mu(Q_{i_0})=\mu_{i_0-1}(Q_{i_0})$. We will prove that in this case we have $\mu=\nu_{i_0-1}$.\\

Suppose, aiming for contradiction, that there exists $f\in K[y]$ such that $\mu(f)>\mu_{i_0-1}(f)$, and choose $f$ of minimal degree among all the polynomials having this property. Since
$$
d_{i_0}(\nu)\leq d_{i_0}(\mu),
$$
we have $\deg_{y}(f)\geq d_{i_0}(\nu)$.\\

Let $f=qQ_{i_0}+r$ be the Euclidean division of $f$ by $Q_{i_0}$. By the minimality of $\deg\ f$, we have $\mu(q)=\mu_{i_0-1}(q)$ and $\mu(r)=\mu_{i_0-1}(r)$.\\

We have $\nu_{i_0-1}(r)=\nu(r)\geq inf\{\nu(f),\nu(qQ_{i_0})\}> inf\{\nu_{i_0-1}(f),\nu_{i_0-1}(qQ_{i_0})\}$.\\
Hence $\nu_{i_0-1}(r)>\nu_{i_0-1}(f)=\nu_{i_0-1}(qQ_{i_0})$.\\
But $\mu(qQ_{i_0})=\nu_{i_0-1}(qQ_{i_0})$ and $\mu(f)>\nu_{i_0-1}(f)$, hence $\mu(f)>\mu(qQ_{i_0})=\mu(r)$, therefore
$\nu_{i_0-1}(qQ_{i_0})=\mu(qQ_{i_0})=\mu(r)=\nu_{i_0-1}(r)$ which is a contradiction.\\

Now suppose that $\mu(Q_{i_0})>\mu_{i_0-1}(Q_{i_0})$. We have $d_{i_0}(\mu)=d_{i_0}(\nu)$. We will prove that
$\mu(Q_{i_0})=\beta_{i_0}(\mu)$. Suppose that there exists a monic polynomial $Q$ such that $\deg_{y}(Q)=d_{i_0}(\mu)$,
$\mu_{i_0-1}(Q)<\mu(Q)$ and $\mu(Q)<\mu(Q_{i_0})$.\\

Write $Q=Q_{i_0}+g$ with $\deg_{y}(g)<d_{i_0}(\mu)$. We have $\mu(Q_{i_0})>\mu(Q)=\mu(g)$. Therefore
$\nu(Q_{i_0})>\mu(Q_{i_0})=\mu(g)=\nu(g)$. Hence $\nu(Q_{i_0})>\nu(g)=\nu(Q)$, in particular,
\[
\beta_{i_0}(\nu)=\nu(Q_{i_0})>\nu(Q)
\]
which contradicts the definition of $\beta_{i_0}(\nu)$.\\

We have $\beta_{i_0}(\mu)=\mu(Q_{i_0})$, $d_{i_0}(\mu)=d_{i_0}(\nu)$ and $\mu_{i_0}=\mu_{Q_{i_0}}$. It remains to prove that $\mu=\mu_{i_0}$.\\

Take any polynomial $f$ in $K[y]$. If $\deg_y(f)<d_{i_0}(\mu)$ then $\mu_{i_0}(f)=\mu(f)$. Suppose that
$\deg_y(f)\ge d_{i_0}(\mu)$ and let $f=qQ_{i_0}+r$ be the Euclidean division of $f$ by $Q_{i_0}$.\\

If $\mu(f)>\mu_{i_0}(f)$ then $\mu(f)>\mu(qQ_{i_0})=\mu(r)$. But $\nu(f)\geq \mu(f)$ and $\nu(r)=\mu(r)$, therefore
$\nu(f)>\nu(r)=\nu(qQ_{i_0})$. Then $\nu(qQ_{i_0})=\mu(qQ_{i_0})$ which is impossible, hence $\mu(f)=\mu_{i_0}(f)$.
\end{proof}

\begin{corollary}\label{maximal} The valuations $\nu$ with $N(\nu)=\infty$ or ($N(\nu)\neq \infty$ and $\beta_{N(\nu)}(\nu)=\infty$) are maximal elements of the set of valuations $\mu$ of $K(y)$ with $\mu(x)=1$.  
\end{corollary}

From the preceding results we also deduce
\begin{remark}\label{structure_mu}
\begin{enumerate}
\item $N(\mu)\leq N(\nu)$ and $D(\mu)\leq D(\nu)$.
\item Either $\mu$ is the $y$-adic valuation with $\mu(y)<\nu(y)$, or there exists $i\in I(\nu)$ such that for each $j\leq i$,
$\mu_j=\nu_j$, $I(\mu)=\{1,\dots,i+1\}$, $\{Q_{j}\}_{j\in I(\mu)}$ is a complete set of key polynomials for $\mu$ and
$\mu=[\nu_i,Q_{i+1},\mu(Q_{i+1})]$.
\item $N(\mu)=N<\infty$ and $\mu$ and $\nu$ have the same sets of first $N$ key polynomials. More precisely, any set
$\{Q_i\}_{i\in\{1,\dots,N\}}$ of first $N$ key polynomials for $\mu$ is also a set of first $N$ key polynomials for $\nu$ and vice versa.
\end{enumerate}
\end{remark}

\subsection{Structure Theorems}
\begin{theorem}\label{infimum}
Let $\tilde{\mu}$ and $\tilde{\nu}$ be two valuations in $\V$. Then there exists an infimum of $\tilde{\mu}$ and $\tilde{\nu}$ (that is, the greatest element that is less than or equal to $\tilde{\mu}$ and $\tilde{\nu}$) in the poset $\V$.
\end{theorem}
\begin{proof}
Fix local coordinates $x$ and $y$ such that $\tilde{\mu}(x)=\tilde{\nu}(x)=1$. Let $\mu$ and $\nu$ be the corresponding valuations on $k(x,y)$ under the correspondence in Theorem (\ref{bijection}).\\

To prove the Theorem, we will prove that the infimum of $\mu$ and $\nu$ exists.\\

First we will define a valuation $\mu\wedge\nu$ and then prove that it is the infimum of $\mu$ and $\nu$.\\

Let $\{\nu_i\}_{i\in I(\nu)}$ and $\{\mu_i\}_{i\in I(\mu)}$ be the truncations associated to $\nu$ and $\mu$ respectively.\\

Suppose first that for each $i\in I(\nu)\cap I(\mu)$ we have $\nu_i=\mu_i$. If $I(\mu)\subseteq I(\nu)$ then $\mu\leq\nu$ and $\mu\wedge\nu=\mu$, otherwise, if $I(\mu)\subset I(\nu)$ then $\nu<\mu$ and $\mu\wedge\nu=\nu$.\\

Now suppose that there exists $i\in I(\nu)\cap I(\mu)$ such that $\nu_i\ne\mu_i$. Let
\[
i_0=inf\{i\in I(\nu)\cap I(\mu)\mid \nu_i\neq\mu_i\}.
\]
Suppose first that $\nu_{i_0}$ and $\mu_{i_0}$ are comparable. Without loss of generality, we may assume that $\mu_{i_0}<\nu_{i_0}$. In this case put $\mu\wedge\nu=\mu_{i_0}$. 
Clearly $\mu_{i_0}\leq \nu$ and $\mu_{i_0}\leq \nu$.\\

Let $\nu'$ be a valuation of $k(x,y)$ such that $\nu'\leq \mu$ and $\nu'\leq \nu$. Since $\nu'\leq \nu$, we have $\nu'(x)=1$.\\
We know from Remark \ref{structure_mu} (1) that $N(\nu')<N(\mu)$ and $N(\nu')<N(\nu)$. Let $\{\nu'_i\}_{i\leq N(\nu')}$ be the truncations associated to $\nu'$.\\
From  Remark \ref{structure_mu} (2) we know that for each $i<N(\nu')$ we have $\nu'_{i}=\mu_i$ and $\nu'_{i}=\nu_i$, therefore $N(\nu')\leq i_0$.\\
We have $\nu'=\nu'_{N(\nu')}\leq \mu_{N(\nu')}\leq\mu_{i_0}=\mu\wedge\nu$.\\

Next, suppose that $\nu_{i_0}$ and $\mu_{i_0}$ are not comparable. In particular, we have $i_0>1$ (since $\mu_1$ and $\nu_1$ are always comparable). Put $\mu\wedge\nu=\mu_{i_0-1}=\nu_{i_0-1}$. Choose $\nu'$ as in the paragraph above and let
$\{\nu'_i\}_{i\leq N(\nu')}$ be the truncations associated to $\nu'$. By Remark \ref{structure_mu} (3) the valuations $\nu'$ and
$\nu$ have the same set of $N(\nu')$ key polynomials, and the valuations  $\nu'$ and $\mu$ have the same set of $N(\nu')$ key polynomials. Therefore if, $N(\nu')=i_0$, we would have $\nu_{i_0}=[\nu_{i_0-1}, Q_{i_0}, \beta_{i_0}(\nu)]$ and 
$\mu_{i_0}=[\mu_{i_0-1}, Q_{i_0}, \beta_{i_0}(\mu)]$. The latter two valuations are comparable, hence $N(\nu')\leq i_0-1$. We have
$\nu'=\nu'_{N(\nu')}\leq \mu_{N(\nu')}\leq\mu_{i_0-1}=\mu\wedge\nu$.
\end{proof}

\begin{theorem}\label{majorant}
Let $\tilde{S}$ be a totally ordered convex subset of $\V$. Then $\tilde{S}$ has a majorant in $\V$.
\end{theorem}
\begin{remark} A short proof of a more general version of this result --- one for rings of arbitrary dimension --- is given in Lemma 3.9 (i) of \cite{N} using elementary properties.
\end{remark}
\begin{proof}
Since we are searching for a majorant, we may assume that $\tilde{S}$ contains $\tilde{\nu}_{\m}$. Since $\tilde{S}$ is totally ordered, we can fix local coordinates $x$ and $y$ such that $\tilde{\nu}(y)\geq\tilde{\nu}(x)=1$ for all $\tilde{\nu}\in \tilde{S}$.\\

By Theorem \ref{bijection}, there exists a totally ordered convex subset $S$ of the set of valuations over $k(x,y)$, satisfying $1=\nu(x)\leq \nu(y)$ for all $\nu\in S$, and such that $S$ contains $\nu_\m$. Also  by Theorem \ref{bijection} the set $\tilde{S}$ has a majorant in $\V$ if and only if the set $S$ has a majorant in the set of valuations  over $k(x,y)$, satisfying $1=\nu(x)\leq \nu(y)$.\\

By Corollary \ref{maximal}, if $S$ contains an element $\nu$ with $N(\nu)=\infty$ or it contains an element $\nu$ with $\beta_{N(\nu)}(\nu)=\infty$ then $S$ has a maximal element. Suppose that $S$ does not contain a maximal element.\\

By Remark \ref{structure_mu} (1), $N(\nu)$ and $D(\nu)$ define increasing functions on $S$. \\

We claim that there exists an initial segment $I\subset \N$ and a set of monic polynomials $\{Q_i\}_{i\in I(S)}$ such that for every valuation $\nu\in S$ the set $\{Q_i\}_{i\in I(\nu)}$ is complete for $\nu$ (the fact that $I\subset \N$ follows from the fact that $S$ does not contain a maximal element).\\
Indeed, take $\nu\in S$, $N\in \N$ and let $\{Q_{i}\}_{i\leq N}$ be a complete set of key polynomials for $\nu$. Let $\nu'\in S$. If
$\nu'<\nu$, then by Remark \ref{structure_mu} (3) the set $\{Q_{i}\}_{i\leq N(\nu')}$ is a complete set of key polynomials for $\nu'$. Otherwise, if $\nu'>\nu$, then, again by Remark \ref{structure_mu} (3), we can add to $\{Q_{i}\}_{i\leq N}$ the key polynomials
$\{Q_{i}\}_{N<i\leq N(\nu)}$ to obtain a complete set of key polynomials for $\nu'$.\\

 Suppose first that $N(\nu)$ is bounded from above. In this case there exists $N\in \N$ with $I(S)=\{1,\dots,N\}$, and a valuation
 $\nu\in S$ with $N(\nu)=N$.\\

The set $\left\{\beta_N(\nu)\mid \nu\in S,\ N(\nu)=N\right\}$ is bounded in $\bar{\R}$. Let $\bar{\beta}$ be a majorant for this set in $\bar{\R}$. If $N=1$, let $\mu$ be the $y$-adic valuation with $\mu(y)=\bar{\beta}$. Otherwise, if $N>1$, let
$\mu=[\nu_{i-1}, Q_N, \bar{\beta}]$. Then $\mu$ is a majorant for $S$.\\

Now suppose that $N(\nu)$ is unbounded in $\N$, that is, $I(S)=\N$.\\

We have $D(\nu)<\infty$ for all $\nu\in S$ since $S$ does not contain a maximal element. Consider the set
$D(S)=\left\{D(\nu)\mid\nu\in S\right\}$. Again, we have two cases, either $D(S)$ has a maximal element $D$, or it is unbounded in $\N$.

Suppose first that $D(S)$ is unbounded in $\N$. For each $f\in K[y]$, put
\[
\mu(f):=max \{\nu(f)\mid\nu\in S\}.
\]
Note first that this maximum is well defined. Indeed, let $f\in K[y]$. Let $\nu\in S$ with
\[
D(\nu)>\deg_y(f).
\]
For every $\nu'\in S$ with $\nu\leq \nu'$ we have $\nu(f)=\nu'(f)$.\\
It is not difficult to verify that $\mu$ is a valuation on $k(x,y)$ and that $\mu$ is a majorant for $S$.\\

Now suppose that $D(S)$ has a maximal element $D$. There exists a cofinal sequence $\{\nu_{i}\}_{i\in I(S)}$ of valuations in $S$ with
$\nu_i=[\nu_{i-1},\ Q_i,\ \nu_{i}(Q_i)]$ for each $i>1$. Therefore, if we write $\beta_{i}=\nu_i(Q_i)$, the value group $\Gamma_i$ of $\nu_i$ is $\beta_0\Z+\dots+\beta_i\Z$, with $\beta_i\in \Q$, by Proposition \ref{truncations_value_groups}.\\

We claim that for every $f\in K[y]$, if there exists $i\in I(S)$ with $\nu_i(f)=\nu_{i+1}(f)$ then $\nu_{j}(f)=\nu_{i}(f)$ for all $j\in I(S)$, $j\geq i$.\\
Indeed, let $i\in I(S)$ be such that $\nu_i(f)=\nu_{i+1}(f)$. By construction, we have
\[
\nu_i(Q_{i+1})<\nu_{i+1}(Q_{i+1})=\nu_{j}(Q_{i+1})\quad\text{ for all }j>i.
\]
Now let $f=qQ_{i+1}+r$ be the Euclidean division of $f$ by $Q_{i+1}$. Since $\nu_{i+1}(f)=\nu_{i}(f)$, we have
$\nu_{i+1}(qQ_{i+1})>\nu_{i+1}(f)=\nu_{i}(r)$. Now for all $j> i$ we have
$$
\nu_{j}(qQ_{i+1})\geq \nu_{i+1}(Q_{i+1})>\nu_{i}(r)=\nu_{j}(r).
$$
Therefore $\nu_{j}(f)$ must be equal to $\nu_{j}(r)=\nu_{i}(r)=\nu_{i}(f)$.\\

If for all $f\in K[y]$ there exists $i\in I(S)$ with $\nu_i(f)=\nu_{i+1}(f)$, we put
\begin{equation}
\mu(f):=\max_{i\in I(S)} \{\nu_i(f)\}.
\end{equation}
Otherwise, if there exists $f\in K[y]$ with
\begin{equation}
\nu_i(f)<\nu_{i+1}(f)\quad\text{ for all }i\in I(S),\label{eq:maximum}
\end{equation}
take $f$ monic of minimal degree, satisfying (\ref{eq:maximum}). We have $\deg_{y}(f)>D$ by definition of the polynomials $Q_i$ and the valuations $\nu_i$. Put $\mu(f)=\infty$. For a polynomials $g\in K[y]$, let $g=qf+r$ be the Euclidean division of $g$ by $f$, and put
$\mu(g)=\max_{i\in I(S)} \{\nu_i(r)\}$. Then $\mu$ is a valuation of $k(x,y)$ which is a majorant for $S$.\\
\end{proof}

\section{Nonmetric Tree Structure on $\V$}\label{trees}

We will now define rooted non-metric trees.

\begin{definition}
A rooted non-metric tree is a poset $(\T , \leq)$ such that:
\begin{itemize}
\item[(T1)] Every set of the form $I_\tau = \{\sigma \in \T |\ \sigma \leq \tau \}$ is isomorphic (as an ordered set) to a real interval.
\item [(T2)] Every totally ordered convex subset of $\T$ is isomorphic to a real interval.
\item [(T3)] Every non-empty subset $S$ of $\T$ has an infimum in $\T$.
\end{itemize}
\end{definition}
Let us consider the following special case of the condition (T3):
\medskip

\noindent(T3$'$) There exists a (unique) smallest element $\tau_0 \in \T$.
\begin{lemma}\label{mult_inf}(Lemma 3.4 \cite{N}) Under the conditions (T1) and (T3$'$), the following conditions are equivalent:
\begin{itemize}
\item[(T3)] Every non-empty subset $S \subset \T$ has an infimum.
\item[(T3$''$)] Given two elements $\tau$, $\sigma \in\T$, the set $\{\tau, \sigma\}$ has an infimum $\tau \wedge \sigma$.
\end{itemize}
\end{lemma}
\begin{definition}
A rooted nonmetric tree $\T$ is complete if every increasing sequence $\{\tau_{i}\}_{i\geq 1}$ in $\T$ has a majorant, that is, an element $\tau_{\infty}$, with $\tau_i\leq\tau_{\infty}$ for every $i$.
\end{definition}

\begin{theorem}\label{tree_structure}
The valuation space $\V$ is a complete nonmetric tree rooted at $\tilde{\nu}_{\m}$.
\end{theorem}
\begin{proof}
\begin{itemize}
\item [(T3$'$)] It is clear that ($\V,\leq$) is a partially ordered set with unique minimal element $\tilde{\nu}_{\m}$.
\item [(T1)] Fix $\tilde{\nu}$ in $\V$, with $\tilde{\nu}>\nu_{\m}$. We will show that the set
$S=\left\{\tilde{\mu}\in\V\mid\nu_{\m}\leq\tilde{\mu}\leq\tilde{\nu}\right\}$ is a totally ordered set isomorphic to an interval in
$\bar\R_{+}$.

Choose local coordinates $x$ and $y$ such that $1=\tilde{\nu}(x)\leq \tilde{\nu}(y)$.

Let $\nu$ be the valuation of $k(x,y)$ corresponding to $\tilde{\nu}$ and let $\{Q_{i}\}_{i\in I(\nu)}$ be a complete sequence of key polynomials for $\nu$. The sequence $\frac{\beta_{i}(\nu)}{d_{i}(\nu)}$ is strictly increasing. If $I(\nu)$ has a maximal element $\ell$, put $I=\left[1,\frac{\beta_\ell(\nu)}{d_{\ell}(\nu)}\right]\subset \bar{\R}$. Otherwise,  put $I=[1,\infty)\subset \bar{\R}$. We will prove that $S$ is isomorphic to $I$ as an ordered set.\\

To each $t\in I$ we will associate a valuation $\tilde{\nu}_{t}$ in $S$. \\

Let $t\in I$. If $t=1$, put $\tilde{\nu}_{t}=\tilde{\nu}_{\m}$. If $I(\nu)$ has a maximal element $\ell$ and $t=\frac{\beta_\ell(\nu)}{d_{\ell}(\nu)}$, put $\tilde{\nu}_{t}=\tilde{\nu}$. \\

Now suppose that $1<t<\frac{\beta_\ell(\nu)}{d_{\ell}(\nu)}$. There exists a unique element $u\in I(\nu)$ such that $\frac{\beta_{u-1}(\nu)}{d_{u-1}(\nu)}<t\leq \frac{\beta_{u}(\nu)}{d_{u}(\nu)}$. Let $\nu_{t}:=[\mu_{u-1},\ Q_{u},\ td_{u}]$ and $\tilde{\nu}_{t}$ the corresponding valuation in $\V$. That the resulting map is a bijection follows from Remark \ref{structure_mu}.

\item [(T2)] By Theorem \ref{majorant}, every totally ordered convex subset $\tilde{S}$ of $\T$ has a majorant in $\T$. With (T3$'$) and (T1) this proves (T2). This also proves that $\T$ is complete.\\

\item [(T3)] is an immediate consequence of Theorem \ref{infimum} and Lemma \ref{mult_inf}.
\end{itemize}
\end{proof}
\begin{remark} Let $(R,\mathfrak m,k)$ and $(R',\mathfrak m',k')$ be two regular two-dimensional local rings such that the residue fields $k$ and $k'$ have the same cardinality. Let $\iota:k\cong k'$ be a bijection between the two fields (as sets, that is, $\iota$ need not be a homomorphism of fields). Using the results of this paper it can be shown that $\iota$ induces a homeomorphism of the respective valuative trees, associated to $R$ and $R'$. Thus, up to homeomorphism, a valuative tree associated to a regular local ring $(R,\mathfrak,k)$ depends only on the cardinality of the residue field $k$. 
\end{remark}

\end{document}